\documentclass{amsart}

\usepackage{enumerate}        
\usepackage{graphicx}         
\usepackage[tight]{subfigure} 

\newtheorem{theorem}{Theorem}
\newtheorem{lemma}[theorem]{Lemma}
\newtheorem{observation}[theorem]{Observation}
\newtheorem{algorithm}[theorem]{Algorithm}

\theoremstyle{definition}
\newtheorem*{defn}{Definition}

\newtheorem*{assumptions}{Assumptions}

\newcommand{\co}{\colon\thinspace}
\newcommand{\kcomp}{\overline{K}}
\newcommand{\knotinfo}{\emph{KnotInfo}}
\newcommand{\lk}[1]{\mathop{lk}(#1)}
\newcommand{\mfd}{\mathcal{M}}
\newcommand{\mfdc}{\mfd^c}
\newcommand{\nbd}[1]{\mathop{nbd}(#1)}

\newcommand{\R}{\mathbb{R}}
\newcommand{\regina}{\emph{Regina}}
\newcommand{\tri}{\mathcal{T}}
\newcommand{\tric}{\tri^c}
\newcommand{\Z}{\mathbb{Z}}

\newcommand{\cpp}{C%
    \nolinebreak\hspace{-.05em}\raisebox{.4ex}{\tiny\bf +}%
    \nolinebreak\hspace{-.10em}\raisebox{.4ex}{\tiny\bf +}}

\begin{document}

\title[A fast branching algorithm for unknot recognition]%
    {A fast branching algorithm for \\ unknot recognition with \\
    experimental polynomial-time behaviour}
\author{Benjamin A.\ Burton}
\address{School of Mathematics and Physics \\
    The University of Queensland \\
    Brisbane QLD 4072 \\
    Australia}
\email{bab@maths.uq.edu.au}
\author{Melih Ozlen}
\address{School of Mathematical and Geospatial Sciences \\
    RMIT University, GPO Box 2476V \\
    Melbourne VIC 3001 \\
    Australia}
\email{melih.ozlen@rmit.edu.au}
\thanks{The first author is supported by the Australian Research Council
    under the Discovery Projects funding scheme (projects
    DP1094516 and DP110101104).}

\begin{abstract}
It is a major unsolved problem as to whether unknot
re\-cog\-ni\-tion---that is,
testing whether a given closed loop in $\R^3$ can be untangled to form a plain
circle---has a polynomial time algorithm.
In practice, trivial knots (which can be untangled) are typically easy to
identify using fast simplification techniques, whereas non-trivial knots
(which cannot be untangled) are more resistant to being conclusively
identified as such.
Here we present the first unknot recognition algorithm which
is always conclusive and, although exponential time in theory, exhibits a
clear polynomial time behaviour under exhaustive experimentation
even for non-trivial knots.

The algorithm draws on techniques from both topology and integer\,/\,linear
programming, and highlights the potential for new applications of techniques
from mathematical programming to difficult problems in low-dimensional
topology.
The exhaustive experimentation covers all $2977$ non-trivial prime knots
with $\leq 12$ crossings.
We also adapt our techniques to the important topological problems of
3-sphere recognition and the prime decomposition of 3-manifolds.
\end{abstract}

\maketitle

%
%

\section{Introduction}

One of the most well-known computational problems in knot theory is
\emph{unknot recognition}:
given a knot $K$ in $\R^3$ (i.e., a closed loop with
no self-intersections),
can it be deformed topologically (without passing through itself)
into a trivial unknotted circle?
If the answer is ``yes'' then $K$ is called a \emph{trivial} knot,
or the \emph{unknot} (as in Figure~\ref{fig-unknots});
if the answer is ``no'' then
$K$ is called a \emph{non-trivial} knot (as in Figure~\ref{fig-knots}).
This simple yes/no decision problem
is deceptively complex: the best known algorithms require
worst-case exponential time, and it is currently a major open problem as to
whether a polynomial time algorithm is possible.

Here we present the first algorithm for unknot recognition that
guarantees a conclusive result \emph{and},
though still worst-case exponential in theory,
behaves in practice like a polynomial-time algorithm
under systematic, exhaustive experimentation.
The algorithm uses an integrated blend of techniques from topology
(normal surfaces and 0-efficiency) and optimisation (integer and
linear programming), and showcases low-dimensional
topology as a new application area in which mathematical programming
can play a pivotal and practical role.

\begin{figure}[tb]
    \centering
    \subfigure[Some trivial knots\label{fig-unknots}]{%
        \qquad\includegraphics[scale=0.63]{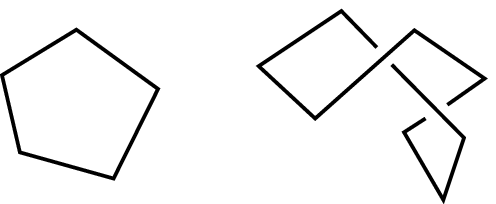}\qquad}
    \subfigure[Some non-trivial knots\label{fig-knots}]{%
        \qquad\includegraphics[scale=0.63]{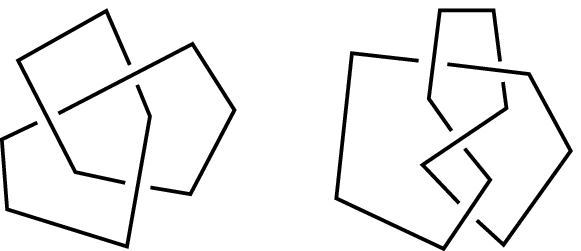}\qquad}
    \caption{Examples of knots in $\R^3$}
    \label{fig-knot}
\end{figure}

The input for unknot recognition is typically a \emph{knot diagram},
i.e., a piecewise-linear projection of the knot $K$ onto the plane
in which line segments ``cross'' over or under one another, as seen in
Figure~\ref{fig-knot}.
The input size is typically measured by the number of crossings
$c$.  This is a reasonable measure, since any $c$-crossing
knot diagram---regardless of how many line segments it uses---can be
deformed into a new diagram of the same
topological knot with just $O(c)$ line segments in total.

Only little is known about the computational complexity of unknot recognition.
The problem is known to lie in \textbf{NP} \cite{hass99-knotnp},
and also in \textbf{co-NP} if the generalised Riemann hypothesis holds
\cite{kuperberg14-conp}.\footnote{%
    Ian Agol gave a talk in 2002 outlining a proof that does not require
    the generalised Riemann hypothesis \cite{agol02-conp},
    but the details are yet to be published.}
Haken's original algorithm from the 1960s
\cite{haken61-knot} has been improved upon by many authors, but the best
derivatives still have worst-case exponential time.  There are alternative
algorithms, such as Dynnikov's grid simplification method
\cite{dynnikov03-knot}, but these are likewise exponential time or
worse.
Nevertheless,
the former results give us reason to believe
that unknot recognition might not be \textbf{NP}-complete,
and nowadays there is increasing discussion as to whether a polynomial
time algorithm could indeed exist \cite{dunfield11-spanning,hass12-conp}.

For inputs that are \emph{trivial} (i.e., topologically equivalent to
the unknot), solving unknot recognition appears to be easy
in practice.  There are widely-available simplification
tools that attempt to ``reduce'' the input to a smaller representation
of the same topological knot in polynomial time
\cite{andreeva02-webservice,regina}, and if they can reduce the
input all the way down to a circle with no crossings then the problem
is solved.
Experimentation suggests that
it is extremely difficult to find
``pathological'' representations of the unknot that do not simplify
in this way \cite{andreeva02-webservice,burton13-regina}.

For input knots that are \emph{non-trivial} (i.e., not equivalent to the
unknot), the situation is more difficult.  Here our simplification tools
cannot help: they might reduce the input somewhat, but we still need to
prove that the resulting knot cannot be completely untangled.
There are many computable knot
invariants that can assist with this task \cite{adams94,lickorish97},
but all known invariants either come with exponential time
algorithms or might lead to inconclusive results (or both).

In this sense, obtaining a ``no'' answer---that is, proving a knot to be
non-trivial---is the more difficult task for unknot recognition.
%
Our new algorithm is fast even in this more difficult case:
when run over the
{\knotinfo} database of all $2977$ prime knots with $\leq 12$ crossings
\cite{www-knotinfo-jun11},
it proves each of them to be non-trivial by solving a \emph{linear} number of
linear programming problems.
Combined with the aforementioned simplification tools
and established polynomial-time algorithms for linear programming
\cite{hacijan79-polylp,karmarkar84-new},
this yields the first algorithm for unknot recognition that guarantees
a conclusive result and \emph{in practice}
exhibits polynomial-time behaviour.

We note that, although our input knots all have
$\leq 12$ crossings, the underlying problems are significantly larger
than this figure suggests---our algorithm works in vector spaces of
dimension up to $350$ for these inputs.  As seen in
section~\ref{s-expt}, both the polynomial profile of our algorithm
and the exponential profile of the prior state-of-the-art algorithms
(against which we compare it)
are unambiguously clear over this data set.

The algorithm is structured as follows.
Given a $c$-crossing diagram of the input knot $K$,
we construct a corresponding 3-dimensional triangulation $\tri$
with $n \in O(c)$ tetrahedra.
We then search this triangulation for a certificate of
unknottedness, using Haken's framework of \emph{normal surface theory}
\cite{haken61-knot}.
The search criteria are deliberately weak, which allows us to perform
the search using a branch-and-bound method (based on integer and
linear programming).
The trade-off is that we might find a ``false certificate'',
which does \emph{not} certify unknottedness; however, in this case we
use our false certificate to shrink the triangulation to fewer
tetrahedra, a process which must terminate after at most $n$ iterations.

The bottleneck in this algorithm is the branch-and-bound phase,
which in the worst case must traverse a search tree of size
$O(3^n \cdot n)$.  However, the experimental performance is far better:
for \emph{every} one of our input knots, the linear programming
relaxations in the branch-and-bound scheme cause the search tree
to collapse to $\sim 8n$ nodes, yielding a polynomial-time search overall.
This is reminiscent of the simplex method for linear programming,
whose worst-case exponential time does not prevent it from
being the tool of choice in many practical applications.

We emphasise that this polynomial time behaviour is measured
purely by experiment.
We do not prove average-case, smoothed or generic complexity results;
though highly desirable, such results are scarce in the study of
3-dimensional triangulations.  We discuss the reasons for this
scarcity in Section~\ref{s-disc}.

Traditional algorithms based on normal surfaces do not use
optimisation; instead they perform an expensive
\emph{enumeration} of extreme rays of a polyhedral cone (see
\cite{burton10-dd,burton13-tree} for the computational details).
Our use of optimisation follows early ideas of Casson and
Jaco et al.\ \cite{jaco02-algorithms-essential}:
essentially we minimise the genus of a surface
in $\R^3$ that the knot bounds, which is zero if and only if the
knot is trivial.
The layout of our branch-and-bound search tree is inspired by earlier work
of the present authors on
normal surface enumeration algorithms \cite{burton13-tree}.

Optimisation approaches to unknot recognition and related problems
have been attempted before, but none have
exhibited polynomial-time behaviour to date.
The key difficulty is that we must
optimise a linear objective function (measuring genus)
over a \emph{non-convex} domain (which encodes potential surfaces).
In previous attempts:
\begin{itemize}
    \item Casson and Jaco et al.\ split the domain
    into an exponential number of convex pieces and perform linear
    programming over each \cite{jaco02-algorithms-essential}.
    This gives a useful upper bound on the running time of
    $O(3^n \times \mathrm{poly}(n))$.
    However, the resulting algorithm is unsuitable
    because for non-trivial input knots the running time also has a
    \emph{lower bound} of $\Omega(3^n \times \mathrm{poly}(n))$
    \cite{burton13-regina}.

    \item In prior work, the present authors express this
    optimisation as an integer program
    using precise search criteria that guarantee to find a
    true certificate if one exists (i.e., the criteria
    are not weakened as described earlier) \cite{burton12-crosscap}.
    The resulting integer programs are extremely difficult to solve
    in exact arithmetic:
    they yield useful bounds for invariants such as knot genus
    and crosscap number, but are unsuitable for decision problems
    such as unknot recognition.
\end{itemize}

Beyond unknot recognition, we also adapt our new algorithm to
the related topological problems of \emph{3-sphere recognition}
and \emph{prime decomposition} of 3-mani\-folds;
see Section~\ref{s-sphere} for details.


\section{Preliminaries}

Normal surface theory is a powerful algorithmic toolkit in
low-dimensional topology: it sits
at the heart of Haken's original unknot recognition algorithm
\cite{haken61-knot}
and provides the framework for the algorithms in this paper.
Here we give a very brief overview of knots, triangulations and normal
surfaces.
For more details on the role this theory plays in
unknot recognition and related topological problems, we refer the reader
to the excellent summary by Hass et~al.\ \cite{hass99-knotnp}.

We consider a \emph{knot} to be a piecewise linear simple closed curve
embedded in $\R^3$, formed from a finite number of line segments.
Two knots $K,K'$ are considered \emph{equivalent} if one can
be continuously deformed into the other without introducing
self-intersections.
Any knot equivalent to a simple planar polygon is said to be
\emph{trivial}, or the \emph{unknot};
all other knots are said to be \emph{non-trivial}.
Again, see Figure~\ref{fig-knot} for examples.

A \emph{knot diagram} is a projection of a knot into the plane with only
finitely many multiple points, each of which is a double point at which
two ``strands'' of the knot cross transversely, one ``passing over'' the other.
These double points are
called \emph{crossings}: the four knot diagrams in Figure~\ref{fig-knot}
have 0, 2, 3 and 4 crossings respectively.  Alternatively, knot diagrams can
be described as annotated 4-valent planar multigraphs;
see \cite{hass99-knotnp} for the details.
A knot diagram with $c$ crossings can (up to knot equivalence) be
described in
$O(c \log c)$ space (see \cite{hoste05-enumeration}
for some examples of encoding schemes),
and in this paper we treat knot diagrams as the usual means by which knots
are presented as input.

By adding a point at infinity, we can extend the ambient space from $\R^3$ to
the topological 3-sphere $S^3 \equiv \R^3 \cup \{\infty\}$.
For any knot $K$, we can then remove an open regular neighbourhood of
$K$ from $S^3$ (essentially ``drilling out'' the knot from $S^3$);
this yields a 3-manifold with torus boundary called the
\emph{knot complement} $\kcomp$.
Given a knot diagram with $c$ crossings, Hass et~al.\ show how to
construct a triangulation of $\kcomp$ with $O(c)$ tetrahedra in
$O(c \log c)$ time \cite[Lemma~7.2]{hass99-knotnp}.

Although the Hass et~al.\ construction produces a simplicial complex,
computational topologists often work with \emph{generalised triangulations},
which are more flexible and often significantly smaller.  A generalised
triangulation begins with $n$ abstract tetrahedra, and affinely identifies
(or ``glues'')
some or all of their $4n$ triangular faces in pairs.
Two different faces of the
same tetrahedron may be glued together; moreover, as a
consequence of the face gluings we may find that multiple edges of the
same tetrahedron become identified together, and likewise with vertices.
Unless otherwise specified, all triangulations in this paper are
generalised triangulations.

Of particular importance are \emph{one-vertex triangulations}, in which
all $4n$ tetrahedron vertices become identified as a single point.
Essentially, these are the 3-dimensional analogues
of well-known constructions in two dimensions, such as the
two-triangle torus and the two-triangle Klein bottle shown in
Figure~\ref{fig-torus-kb} (the different arrowheads indicate how
edges are glued together)---both of these
2-dimensional examples are one-vertex triangulations also.

\begin{figure}[tb]
    \centering
    \subfigure[A torus\label{fig-torus}]{%
        \qquad\includegraphics[scale=0.7]{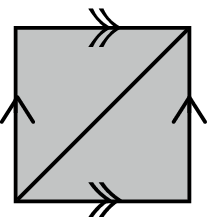}\qquad}%
    \quad
    \subfigure[A Klein bottle\label{fig-kb}]{%
        \qquad\includegraphics[scale=0.7]{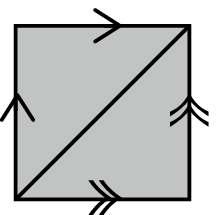}\qquad}
    \caption{Triangulated surfaces in two dimensions}
    \label{fig-torus-kb}
\end{figure}

If a triangulation $\tri$
represents some underlying 3-manifold $\mfd$,
then those tetrahedron faces of $\tri$ that are \emph{not} glued to a partner
together form a triangulated surface (possibly empty, possibly disconnected)
which we call the \emph{boundary} of $\tri$, denoted by $\partial \tri$;
topologically this represents the 3-manifold boundary $\partial \mfd$.

A \emph{normal surface} in $\tri$ is a surface $S$ that is properly embedded
(i.e., embedded so that $\partial S = S \cap \partial \tri$),
and which meets each tetrahedron of $\tri$ in a
(possibly empty) collection of curvilinear triangles and
quadrilaterals, as illustrated in Figure~\ref{fig-normaldiscs}.
In each tetrahedron these triangles and quadrilaterals are classified into
seven \emph{types} according to which edges of the tetrahedron they
meet, as illustrated in Figure~\ref{fig-normaltypes}:
four \emph{triangle types} (each ``truncating'' one of the four vertices),
and three \emph{quadrilateral types} (each separating the four
vertices into two pairs).

\begin{figure}[tb]
    \centering
    \subfigure[The surface $S$ meeting a
        single tetrahedron\label{fig-normaldiscs}]{%
        \qquad\includegraphics[scale=0.4]{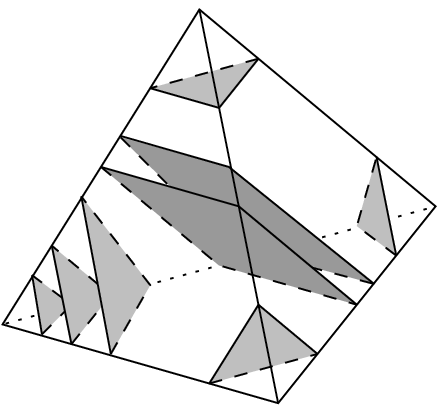}\qquad}
    \qquad\qquad
    \subfigure[The seven normal disc types\label{fig-normaltypes}]{%
        \includegraphics[scale=0.35]{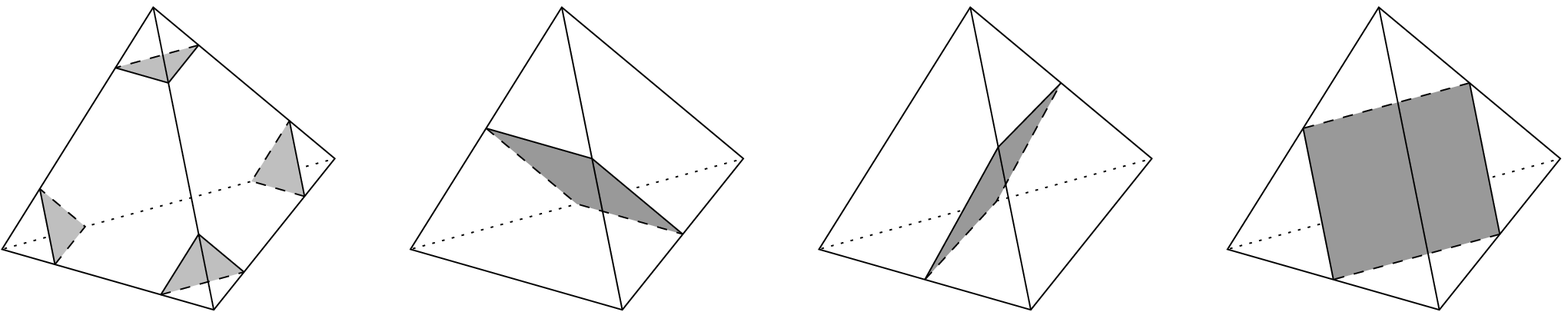}}
    \caption{Normal triangles and quadrilaterals}
\end{figure}

Any normal surface $S$ in an $n$-tetrahedron triangulation $\tri$ can be
described by a vector of $7n$ non-negative integers that counts the number
of triangles and quadrilaterals of each type in each tetrahedron;
we denote this by $\mathbf{v}(S) \in \Z^{7n}$.
The individual coordinates of $\mathbf{v}(S)$ that count triangles and
quadrilaterals are referred to as \emph{triangle} and
\emph{quadrilateral coordinates} respectively.
This vector uniquely identifies the surface (up to a certain class of
isotopy).  More generally, a result of Haken \cite{haken61-knot} shows
that an integer vector $\mathbf{x} \in \R^{7n}$
represents a normal surface if and only if:

\begin{enumerate}
    \item $\mathbf{x} \geq 0$;
    \item $A\mathbf{x}=0$, where $A$ is a matrix of up to $6n$
    linear \emph{matching equations} derived from the specific
    triangulation $\tri$;
    \item $\mathbf{x}$ satisfies the \emph{quadrilateral constraints},
    a collection of $n$ combinatorial constraints (one per tetrahedron)
    that require, for each tetrahedron, at most one of the three
    corresponding quadrilateral coordinates to be non-zero.
\end{enumerate}

In essence, the matching equations ensure that normal triangles and
quadrilaterals can be glued together across adjacent tetrahedra,
and the quadrilateral constraints ensure that
quadrilaterals within the same tetrahedron can avoid intersecting.
Any vector $\mathbf{x} \in \R^{7n}$ (real or integer) that
satisfies all three of these conditions is called \emph{admissible}.

The Euler characteristic of a surface, denoted $\chi$,
is a topological invariant that essentially encodes the genus of the surface.
In particular, an orientable surface of genus $g$ with $b$
boundary curves has Euler characteristic $\chi = 2-2g-b$,
and a non-orientable surface of genus $g$ with $b$
boundary curves has Euler characteristic $\chi = 2-g-b$.
Given any polygonal decomposition of a surface, its Euler characteristic
can be computed
as $\chi = V-E+F$, where $V$, $E$ and $F$ count vertices, edges and
2-faces respectively.

For normal surfaces within a fixed $n$-tetra\-hedron triangulation $\tri$,
their Euler characteristics can be expressed
using a homogeneous \emph{linear} function on normal coordinates.
That is, there is a homogeneous linear function $\chi \co \R^{7n} \to \R$
such that, if $S$ is any normal surface in $\tri$, then
$\chi(\mathbf{v}(S))$ is the Euler characteristic of $S$.
In fact there are many choices for such a function;
see \cite{burton12-crosscap, jaco95-algorithms-decomposition}
as well as the proof of Lemma~\ref{l-search-broad} in this paper
for various formulations.

\begin{figure}[tb]
    \centering
    \includegraphics[scale=0.6]{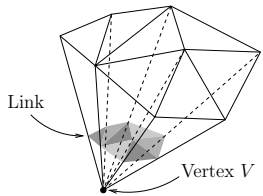}
    \caption{Building a vertex link from triangles}
    \label{fig-link}
\end{figure}

Let $V$ be a vertex of some triangulation $\tri$.  Then the
\emph{link} of $V$ is the frontier of a small regular neighbourhood of $V$.
In a 3-manifold triangulation, each vertex link is either a disc
(for a boundary vertex $V \in \partial \tri$) or a sphere
(for an internal vertex $V \notin \partial \tri$).
The link can be presented as a normal surface built from triangles
only (see Figure~\ref{fig-link}),
and if $\tri$ is a one-vertex triangulation then this normal surface
contains precisely one triangle of each type.

\begin{figure}[tb]
    \centering
    \includegraphics[scale=0.25]{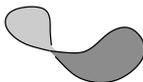}
    \caption{A disc bounded by a trivial knot}
    \label{fig-unknot-disc}
\end{figure}

Haken's original unknot recognition algorithm is based on the
observation that any trivial knot must bound an embedded disc in $\R^3$
(see Figure~\ref{fig-unknot-disc}).  In the knot complement
$\kcomp$, this corresponds to a properly embedded disc in $\kcomp$
that meets the boundary torus $\partial \kcomp$ in a non-trivial curve
(i.e., a curve that does not bound a disc in $\partial \kcomp$).
Moreover, we have:

\begin{theorem}[Haken] \label{t-haken}
    Let $K$ be a knot and let $\tri$ be a triangulation of the
    complement $\kcomp$.  Then $K$ is trivial if and only if
    $\tri$ contains a \emph{normal} disc whose boundary is a non-trivial
    curve in $\partial \tri$.
\end{theorem}


\section{The algorithm} \label{s-alg}

In this paper we describe the new algorithm at a fairly high level.
For details of the underlying data structures, we refer the
reader to the thoroughly documented source code for the software
package {\regina} \cite{regina}, in which this algorithm is
implemented.\footnote{%
    See in particular the routine \texttt{NTriangulation::isSolidTorus()}.}

At the highest level, the new algorithm operates as follows.

\begin{algorithm}[Unknot recognition] \label{alg-unknot}
To test if an input knot $K$ is trivial, given a
knot diagram of $K$:
\begin{enumerate}
    \item \label{en-alg-tri}
    Build a triangulation of the knot complement $\kcomp$.

    \item \label{en-alg-simplify}
    Make this into a one-vertex triangulation $\tri$ of $\kcomp$
    without increasing the number of tetrahedra.
    Let the number of tetrahedra in $\tri$ be $n$.

    \item \label{en-alg-search}
    Search for a connected normal surface $S$ in $\tri$ which is not the vertex
    link and has positive Euler characteristic.

    \begin{itemize}
    \item
    If no such $S$ exists, then the knot $K$ is non-trivial.

    \item
    If $S$ is a disc whose boundary is a non-trivial curve in
    $\partial \tri$, then the knot $K$ is trivial.

    \item
    Otherwise we modify $\tri$ by crushing the surface $S$
    to obtain a new triangulation $\tri'$ of
    $\kcomp$ with fewer than $n$ tetrahedra,
    and return to step~\ref{en-alg-simplify} using this
    new triangulation $\tri'$ instead.
    \end{itemize}
\end{enumerate}
\end{algorithm}

To build the initial triangulation of $\kcomp$
in step~\ref{en-alg-tri}, we use the
method of Hass et~al.\ \cite[Lemma~7.2]{hass99-knotnp}
as noted earlier in the preliminaries section.
For the subsequent steps, we give details in separate sections below:
\begin{itemize}

    \item To make a one-vertex triangulation
    in step~\ref{en-alg-simplify}, we combine the Jaco-Rubin\-stein
    crushing procedure (described in Section~\ref{s-jr}) with a combinatorial
    expansion procedure (described in Section~\ref{s-onevtx}).

    \item To find the surface $S$ in step~\ref{en-alg-search}, we use
    techniques from integer and linear programming
    (described in Section~\ref{s-search}).
    This search is the main bottleneck---and hence the most critical
    component---of the algorithm.

    \item To crush $S$ in step~\ref{en-alg-search}, we once more use the
    Jaco-Rubinstein procedure (again see Section~\ref{s-jr}).
\end{itemize}

After presenting these details, in Section~\ref{s-alg-proofs}
we prove the algorithm correct and analyse its
time complexity.  In particular, we show that the search for $S$
in step~\ref{en-alg-search} is the only potential source of
exponential time.  That is, if the search for $S$
exhibits polynomial time in practice (as we quantify experimentally
in Section~\ref{s-expt}), then this translates to
polynomial-time behaviour in practice for the algorithm as a whole.


\subsection{The Jaco-Rubinstein crushing procedure} \label{s-jr}

Steps~\ref{en-alg-simplify} and \ref{en-alg-search} of the algorithm
above make use of the \emph{crushing procedure} of Jaco and Rubinstein,
which modifies a triangulation by ``destructively'' crushing a normal
surface within it.\footnote{%
    It is clear which surfaces we crush in step~\ref{en-alg-search}
    of the main algorithm.  In step~\ref{en-alg-simplify} we crush temporary
    surfaces that we create on the fly; see Section~\ref{s-onevtx}
    for the details.}
Here we outline this procedure, and then prove in Lemma~\ref{l-crush}
that (i)~in our setting it yields a smaller triangulation
of the knot complement $\kcomp$,
and (ii)~this smaller triangulation can be obtained in linear time.

Our outline of the crushing procedure is necessarily brief.
See the original Jaco-Rubinstein paper \cite{jaco03-0-efficiency}
for the full details, or \cite{burton14-crushing-dcg}
for a simplified approach.

Mathematically, the crushing procedure operates as follows.
Given a normal surface $S$ within a triangulation $\tri$, we
first cut $\tri$ open
along $S$ and then collapse all of the triangles and quadrilaterals of $S$
(which now appear twice each on the boundary) to points,
as shown in Figure~\ref{fig-crushbdry}.
This results
in a new cell complex that is typically not a triangulation,
but instead is built from a combination of
tetrahedra, footballs and/or pillows as
illustrated in Figure~\ref{fig-jrpieces}.  We then flatten each football
to an edge and each pillow to a triangular face, as shown in
Figure~\ref{fig-jrcrush}, resulting in a new triangulation $\tric$
that is the final result of the crushing procedure.
In general this final triangulation $\tric$
might be disconnected, might represent a
different 3-manifold from $\tri$, or might not even represent a
3-manifold at all.
We now make two important observations, both due to Jaco and Rubinstein
\cite{jaco03-0-efficiency}:

\begin{figure}[tb]
    \centering
    \subfigure[Cutting along and then collapsing
        a surface\label{fig-crushbdry}]{%
        \includegraphics[scale=0.58]{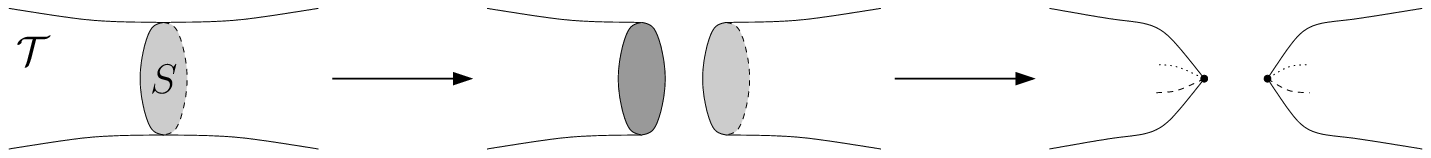}}
    \subfigure[Intermediate pieces\label{fig-jrpieces}]{%
        \includegraphics[scale=0.28]{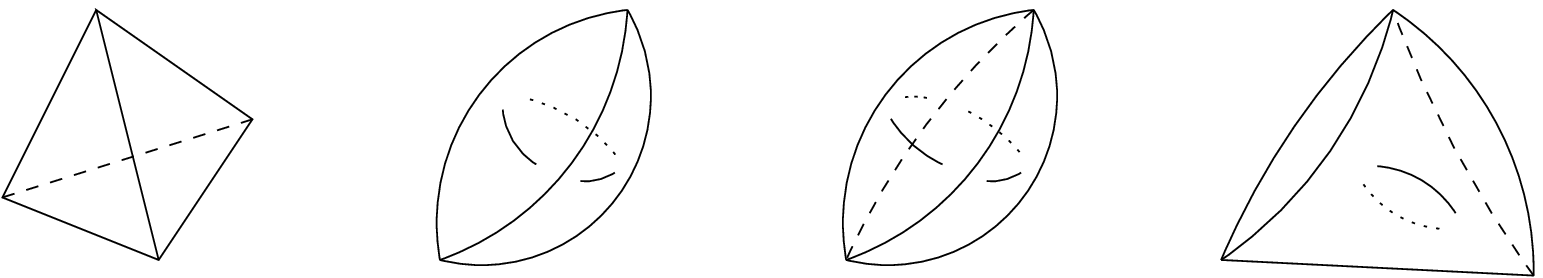}}
    \qquad\qquad
    \subfigure[Flattening footballs and pillows\label{fig-jrcrush}]{%
        \includegraphics[scale=0.47]{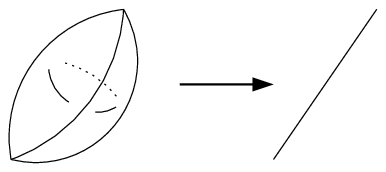}\quad
        \includegraphics[scale=0.47]{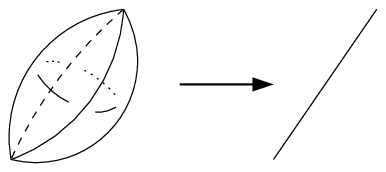}\quad
        \includegraphics[scale=0.47]{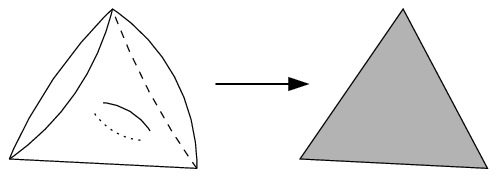}}
    \caption{The Jaco-Rubinstein crushing procedure}
    \label{fig-jrcrushall}
\end{figure}

\begin{observation} \label{o-crush-shrink}
    The tetrahedra in $\tric$ (i.e., those that ``survive'' the
    crushing process) correspond precisely to those
    tetrahedra in $\tri$ that do not contain any quadrilaterals of $S$.
    In particular, unless $S$ is a collection of vertex links, the number of
    tetrahedra in $\tric$ will be strictly smaller than in $\tri$.
\end{observation}

\begin{observation} \label{o-crush-change}
    If $S$ is a normal sphere or disc and $\tri$ triangulates an
    orientable 3-manifold $\mfd$ (with or without boundary), then
    the result $\tric$ triangulates some new
    3-manifold $\mfdc$ (possibly disconnected or empty) that is
    obtained from $\mfd$ by zero or more of the following operations:
    \begin{enumerate}[(i)]
        \item cutting $\mfd$ open along spheres and filling the
        resulting boundary spheres with 3-balls;
        \item cutting $\mfd$ open along properly embedded discs;
        \item capping boundary spheres of $\mfd$ with 3-balls;
        \item deleting entire connected components that are any of
        the 3-ball, the
        3-sphere, projective space $\R P^3$, the lens space $L_{3,1}$
        or the product space $S^2 \times S^1$.
    \end{enumerate}
\end{observation}

See \cite[Theorem~2]{burton14-crushing-dcg} for a simplified treatment and proof
of Observation~\ref{o-crush-change}.

Our result below uses both of these observations to show that
crushing does what we need in our setting.  Moreover, it shows through
amortised complexity arguments that, with the right
implementation, crushing can be carried out in linear time.

\begin{lemma} \label{l-crush}
    Given a triangulation $\tri$ of a knot complement $\kcomp$
    with $n>1$ tetrahedra and a connected normal surface $S$
    in $\tri$ which is not a vertex link and which
    has positive Euler characteristic,
    we can crush $S$ using the Jaco-Rubinstein
    procedure and convert the result into a new triangulation
    $\tri'$ of $\kcomp$ with strictly fewer than $n$ tetrahedra,
    all in $O(n)$ time.
\end{lemma}

\begin{proof}
    The only connected surfaces with positive Euler characteristic are
    the sphere, the disc, and the projective plane.  Since $\tri$
    triangulates a knot complement in $S^3$ and $S^3$ does not contain
    an embedded projective plane, we conclude that
    our surface $S$ is either a
    sphere or a disc.

    We first show how to build a smaller triangulation $\tri'$ of $\kcomp$.
    We begin by crushing $S$ in $\tri$ to obtain an intermediate
    triangulation $\tric$.
    Since $\kcomp$ is a knot complement, the list of possible
    topological changes from Observation~\ref{o-crush-change}
    becomes much simpler:
    \begin{enumerate}[(i)]
        \item Any embedded sphere in a knot complement bounds a ball, and so
        the first operation (cutting along spheres and filling
        the boundaries with balls) is equivalent to adding new
        (disconnected) 3-sphere components.

        \item The result of cutting along a properly embedded disc $D$
        depends on the boundary of this disc.
        If the boundary of $D$ is \emph{trivial}
        on the torus $\partial \kcomp$, then since $\kcomp$ is a knot complement
        it must be true that $D$ and some portion of $\partial \kcomp$
        together bound a 3-ball, and cutting along $D$
        simply adds a new 3-ball component.
        If the boundary of $D$ is \emph{non-trivial}
        on $\partial \kcomp$ then $K$ must be a
        trivial knot, and cutting along $D$ converts the entire manifold
        $\kcomp$ into the 3-ball.

        \item The only boundary spheres that might occur will be the
        boundaries of new 3-ball components created in the previous step.
        Capping these spheres with 3-balls has the effect of converting
        these new 3-ball components into new 3-sphere components instead.

        \item If we begin with a knot complement, none of the operations
        above can ever produce the manifolds
        $\R P^3$, $L_{3,1}$ or $S^2 \times S^1$.
        Therefore the only connected components that we could ever delete
        are 3-balls and 3-spheres.
    \end{enumerate}

    In other words, the topological changes introduced in the
    intermediate triangulation $\tric$ are limited to
    possibly replacing $\kcomp$ with a 3-ball (but only if the knot
    $K$ is trivial),
    and adding or removing 3-ball and/or 3-sphere components.

    It is clear now how to obtain the new triangulation $\tri'$ of $\kcomp$.
    If any connected component of the intermediate triangulation $\tric$
    has torus boundary then we can take this component as the new
    triangulation $\tri'$.  Otherwise we know that $K$ must be the trivial
    knot,
    and we can take $\tri'$ to be the standard one-tetrahedron triangulation
    of the solid torus \cite{jaco03-0-efficiency} (though we
    could of course just terminate the unknot recognition algorithm
    immediately).  Either way, $\tri'$ triangulates the
    same knot complement $\kcomp$.

    Because $S$ is not the vertex link, it contains at least one
    quadrilateral and so crushing will strictly reduce the number of
    tetrahedra (recall Observation~\ref{o-crush-shrink} above).
    It is clear then that after extracting the component
    with torus boundary (or building a new one-tetrahedron solid torus),
    the number of tetrahedra in $\tri'$ will be strictly less than $n$.

    We now show how to obtain this smaller triangulation $\tri'$ in
    $O(n)$ time.  There are two procedures that we must analyse:
    (i)~the Jaco-Rubinstein crushing procedure (converting $\tri \to \tric$),
    and (ii)~extracting the connected component with torus boundary
    if one exists (converting $\tric \to \tri'$).

    \begin{itemize}
        \item \emph{The crushing procedure (converting $\tri \to \tric$):}

        It is simple to identify
        which tetrahedra of $\tri$ survive the crushing procedure (those whose
        corresponding quadrilateral coordinates in $S$ are all zero,
        as in Observation~\ref{o-crush-shrink})---the
        main challenge is to identify in $O(n)$ time how the faces of
        these surviving tetrahedra are to be glued together.

        Let $\phi$ be some face of a surviving tetrahedron $\Delta$.
        To identify the new partner face for $\phi$, we trace a path
        through adjacent tetrahedra in the original triangulation
        $\tri$ as follows.
        Whenever we enter a tetrahedron that contains quadrilaterals of $S$,
        we cross to the face on the opposite side of these quadrilaterals
        (as depicted in Figure~\ref{fig-quadcross})
        and continue through to the next adjacent tetrahedron.
        If we ever reach a tetrahedron with no quadrilaterals of $S$, then the
        resulting face $\phi'$ is the partner to which $\phi$ is glued.
        If instead we reach the boundary $\partial \tri$, then $\phi$
        becomes a boundary face of the new triangulation.
        See Figure~\ref{fig-quadcycle} for a full illustration.

        \begin{figure}[tb]
            \centering
            \subfigure[Crossing to the opposite side of the quadrilaterals]{%
                \label{fig-quadcross}%
                \quad\includegraphics[scale=0.9]{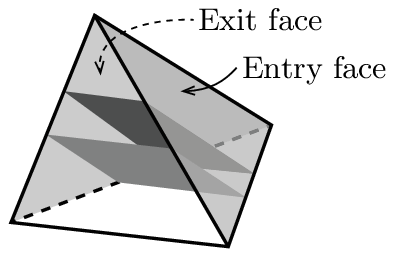}\quad}
            \qquad
            \subfigure[Moving from $\phi$ to the partner face $\phi'$]{%
                \label{fig-quadcycle}%
                \qquad\includegraphics[scale=0.9]{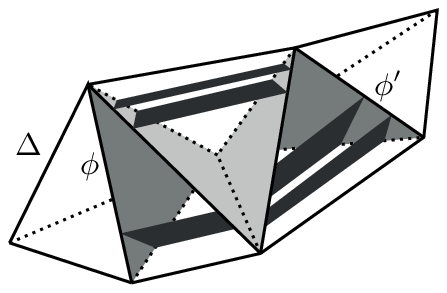}\qquad}
            \caption{Identifying the face gluings after crushing}
        \end{figure}

        To see why following paths in this way correctly identifies the face
        gluings, we must study the way in which pillows (formed on
        either side of the quadrilaterals in a tetrahedron) flatten to
        triangles, as seen earlier in Figure~\ref{fig-jrcrush}.
        Jaco and Rubinstein explain this in their original paper
        \cite{jaco03-0-efficiency}, and we do not reiterate the details here.
        It is important to note that such a path cannot cycle (since
        the starting tetrahedron $\Delta$ does not contain quadrilaterals),
        and so we must eventually reach one of the two conclusions described
        above.

        Regarding time complexity: although any individual path may be
        long, we note that each tetrahedron of $\tri$ can feature at
        most twice amongst all the paths (once for each of the two
        directions in which we might cross over the quadrilaterals).
        Therefore the total length of \emph{all} such paths is $O(n)$.
        Each step in such a path takes $O(1)$ time to compute (since all we
        need to know is which, if any, of the three quadrilateral coordinates is
        non-zero in the next tetrahedron), and so the gluings in the
        crushed triangulation $\tric$ can all be computed in total $O(n)$ time.

        \medskip

        \item
        \emph{Extracting the component with torus boundary
        (converting $\tric \to \tri'$):}

        Since each tetrahedron is adjacent to at most four
        neighbours, we can identify connected components of $\tric$
        in $O(n)$ time
        by following a depth-first search through adjacent tetrahedra
        for each component.
        This takes $O(t)$ time for each
        $t$-tetrahedron component, summing to $O(n)$ time overall.

        For each connected component, we ``build the skeleton'';
        that is, group the $4t$ vertices,
        $6t$ edges and $4t$ faces of the $t$ individual tetrahedra into
        equivalence classes that indicate how these are identified
        (or ``glued together'') in the triangulation.
        As before, we do this by following a depth-first
        search through adjacent tetrahedra for each equivalence class.
        Each search requires time proportional to the size of the
        equivalence class, summing to $O(t)$ time for each component
        and $O(n)$ time overall.

        For components with non-empty boundary, we can test
        for \emph{torus} boundary by counting equivalence classes of
        vertices, edges and faces, and computing the Euler characteristic
        $\chi = \mathrm{vertices}-\mathrm{edges}+\mathrm{faces}-
        \mathrm{tetrahedra}$.
        This is enough to distinguish between components with
        torus boundary ($\chi=0$) versus sphere boundary ($\chi=1$),
        which from earlier are the
        only possible scenarios after crushing.
        Once again this sums to $O(n)$ time overall.
        \qedhere
    \end{itemize}
\end{proof}


\subsection{Building a one-vertex triangulation} \label{s-onevtx}

In step~\ref{en-alg-simplify} of Algorithm~\ref{alg-unknot},
we convert our existing triangulation of $\kcomp$ into a
one-vertex triangulation of $\kcomp$.
There are well-known algorithms for producing one-vertex triangulations
of a 3-manifold \cite{jaco03-0-efficiency,matveev90-complexity}, though
they have not been studied from a complexity viewpoint.  We give a
method that combines Jaco-Rubinstein crushing with a tightly-controlled
subcomplex expansion technique,
and show that it runs in small polynomial time.

Before we formally state and prove this result, we introduce some
terminology.

\begin{defn}
    Let $\mathcal{E}$ be a subcomplex of a triangulation $\tri$;
    that is, a union of vertices, edges, triangles
    and/or tetrahedra of $\tri$.
    Then by the \emph{neighbourhood of $\mathcal{E}$}, denoted
    by $\nbd{\mathcal{E}}$, we mean the closure of a
    small regular neighbourhood of $\mathcal{E}$.
    By the \emph{link of $\mathcal{E}$}, denoted by $\lk{\mathcal{E}}$,
    we refer to the frontier of this neighbourhood.
\end{defn}

\begin{figure}[tb]
    \centering
    \subfigure[A subcomplex $\mathcal{E}$]{%
        \qquad\includegraphics[scale=0.45]{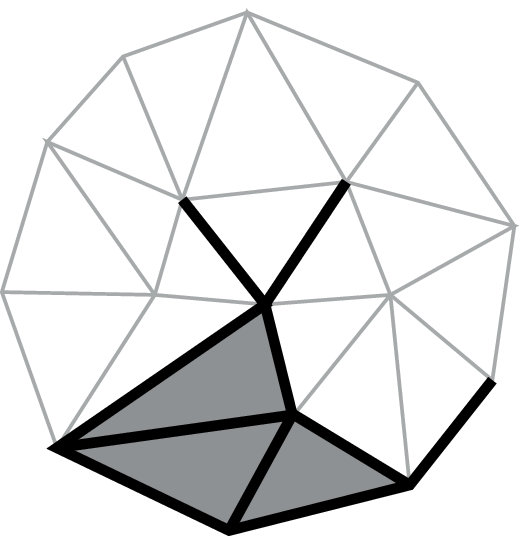}\qquad}
    \subfigure[The neighbourhood $\nbd{\mathcal{E}}$]{%
        \quad\qquad\includegraphics[scale=0.45]{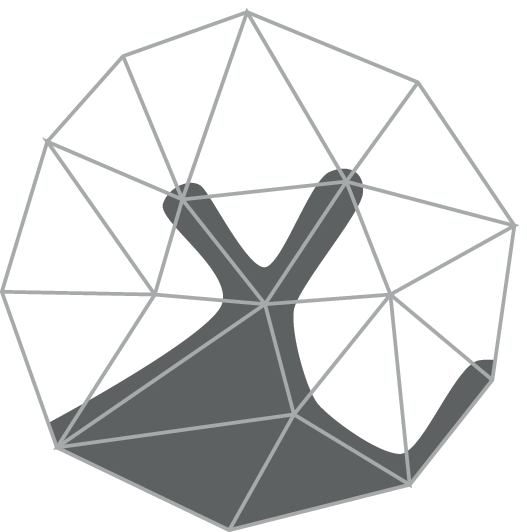}\qquad\quad}
    \subfigure[The link $\lk{\mathcal{E}}$]{%
        \qquad\includegraphics[scale=0.45]{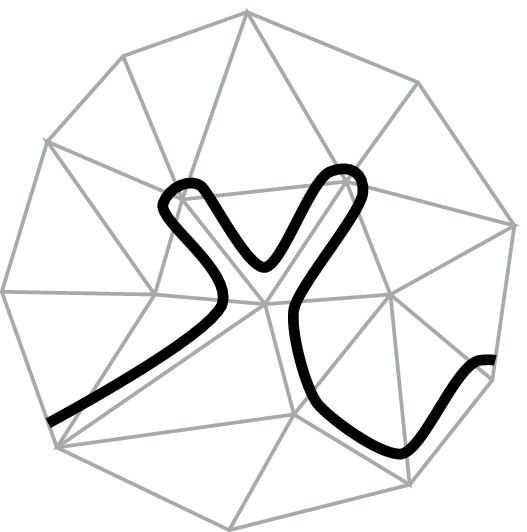}\qquad}
    \caption{Neighbourhood and link in a 2-dimensional triangulation}
    \label{fig-nbdlink}
\end{figure}

Figure~\ref{fig-nbdlink} illustrates both concepts in the 2-dimensional
setting, using a triangulation of the disc.
In our 3-dimensional setting, $\nbd{\mathcal{E}}$ is always a 3-manifold
with boundary, and $\lk{\mathcal{E}}$ (which generalises the earlier
concept of a vertex link) is always a properly embedded surface in $\tri$.

\begin{lemma} \label{l-onevtx}
    Given an $n$-tetrahedron triangulation $\tri$ of a
    knot complement $\kcomp$,
    we can construct a one-vertex triangulation
    $\tri'$ of $\kcomp$ with at most $n$ tetrahedra in $O(n^3)$ time.
\end{lemma}

\begin{proof}
    First we count the number of vertices of $\tri$.
    We can do this in $O(n)$ time by partitioning
    the $4n$ individual tetrahedron
    vertices into equivalence classes that indicate how they are
    identified in the overall triangulation, using a depth-first
    search through the gluings between adjacent tetrahedra.

    If $\tri$ has only one vertex, then we are finished.
    Therefore we assume from here on that $\tri$ has more than one vertex.
    It follows immediately that we have $n \geq 2$ tetrahedra,
    since by a simple enumeration of all one-tetrahedron
    triangulations, the only one-tetrahedron knot complement is the standard
    one-tetrahedron, one-vertex triangulation of the solid torus
    \cite{jaco03-0-efficiency}.

    Our next task is to locate an edge $e$ that joins two distinct vertices
    of $\tri$.
    If $\tri$ contains more than one vertex
    on the boundary then we choose $e$ to lie entirely in $\partial \tri$
    (as opposed to cutting through the interior of $\tri$);
    otherwise we choose $e$ to join the single boundary vertex with
    some internal vertex (as opposed to joining two internal vertices).
    We can find such an edge in $O(n)$ time simply by iterating
    through all $6n$ individual tetrahedron edges, and observing how
    their endpoints sit within the partition of vertices that we made before.

    Our plan is to convert this edge $e$ into a normal disc that
    is not a vertex link, and to do this in $O(n^2)$ time.
    Given such a normal disc, Lemma~\ref{l-crush} shows that we can crush it
    in $O(n)$ time to obtain a new triangulation of $\kcomp$ with
    strictly fewer tetrahedra.  We repeat this process until we arrive
    at a one-vertex triangulation; since the number of tetrahedra
    decreases at every stage, the process must terminate after at most
    $n$ iterations, giving a running time of $O(n^3)$ overall.

    The remainder of this proof shows how, given our edge $e$,
    we can build a non-vertex-linking normal disc in $O(n^2)$ time.

    From our choice of $e$, the link $\lk{e}$ is already a disc;
    however, it might not be normal.
    For instance, if $e$ appears twice around some triangle of $\tri$,
    then the link of $e$ will meet this triangle in a
    ``bent'' arc (shown in Figure~\ref{fig-bentarc}), which a normal
    surface cannot contain.

    \begin{figure}[tb]
        \centering
        \includegraphics[scale=0.9]{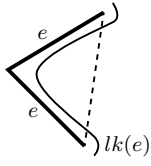}
        \caption{A non-normal edge linking surface}
        \label{fig-bentarc}
    \end{figure}

    Our strategy then is to expand $e$ to a subcomplex $\mathcal{E}$
    of $\tri$ whose link $\lk{\mathcal{E}}$ is a
    \emph{normal} surface, and to show that some component of this
    link is the disc that we seek.  To do this, we initialise
    $\mathcal{E}$ to the single edge $e$ and repeat the following
    expansion steps for as long as possible:
    \begin{itemize}
        \item If any triangle has either two or all three of its
        edges in $\mathcal{E}$, we expand $\mathcal{E}$ to include
        the entire triangle (as illustrated in Figure~\ref{fig-expand});

        \item If any tetrahedron has all four of its triangular
        faces in $\mathcal{E}$,
        we expand $\mathcal{E}$ to include the entire tetrahedron.
    \end{itemize}

    \begin{figure}[tb]
        \centering
        \includegraphics[scale=0.9]{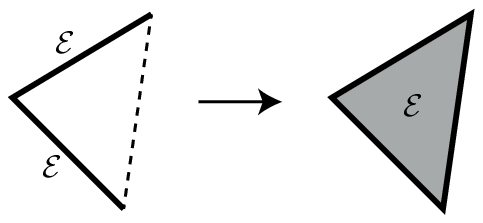}
        \qquad\qquad\qquad
        \includegraphics[scale=0.9]{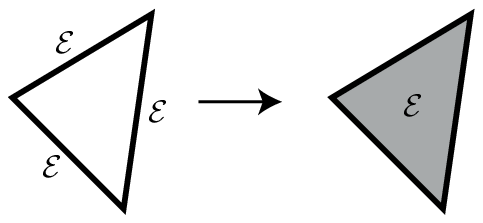}
        \caption{Expanding the subcomplex $\mathcal{E}$}
        \label{fig-expand}
    \end{figure}

    We now present a series of claims that together show that some
    component of $\lk{\mathcal{E}}$ is a non-vertex-linking normal disc.
    After this, we conclude the proof by showing how this disc is
    constructed in $O(n^2)$ time.

    \begin{figure}[tb]
        \centering
        \includegraphics[scale=1.0]{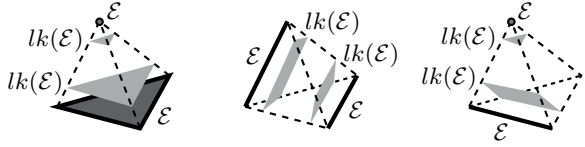}
        \caption{Building the link $\lk{\mathcal{E}}$ as a normal surface}
        \label{fig-normalfrontier}
    \end{figure}

    \textbf{Claim A:} \emph{Once the expansion is finished,
    the link $\lk{\mathcal{E}}$ is normal.}

    To show this, we explicitly construct $\lk{\mathcal{E}}$
    by inserting the following triangles and quadrilaterals into each
    tetrahedron $\Delta$ of $\tri$
    (see Figure~\ref{fig-normalfrontier} for illustrations):

    \begin{itemize}
        \item If $\mathcal{E}$ contains the
        entire tetrahedron $\Delta$, then we do not place any
        triangles or quadrilaterals in $\Delta$.

        \item Otherwise, $\mathcal{E}$ contains at most one triangular
        face of $\Delta$ (since two or more faces would cause the expansion
        process to consume $\Delta$ completely).
        If $\mathcal{E}$ does contain a face of $\Delta$,
        then we place a normal triangle beside this face.

        \item If $\mathcal{E}$ does not contain any faces of $\Delta$,
        it might still contain either one edge of $\Delta$, or two
        opposite edges of $\Delta$ (any more would again cause further
        expansion).  If $\mathcal{E}$ contains such edge(s), then we
        place a normal quadrilateral beside each edge.

        \item If $\mathcal{E}$ contains any additional vertices of
        $\Delta$ that are not yet accounted for, we place a triangle
        next to each such vertex.
    \end{itemize}

    When connected together, these triangles and quadrilaterals form the link
    $\lk{\mathcal{E}}$, thereby establishing claim~A.

    \textbf{Claim B:}
    \emph{The neighbourhood $\nbd{\mathcal{E}}$ is a 3-ball, possibly
    with punctures, and $\nbd{\mathcal{E}}$ meets the boundary $\partial \tri$
    in a non-empty connected region.}

    We prove this by induction, showing that claim~B holds true at each
    stage of the expansion process.
    \begin{itemize}
        \item At the beginning of the expansion process we have
        $\mathcal{E}=e$.  Here our claim is true because we chose $e$
        to join two distinct vertices, and because we chose $e$ to meet
        $\partial \tri$ in either (i)~just one endpoint, or
        (ii)~the entire edge.
        Either way, $\nbd{e}$ is just a 3-ball that meets
        $\partial \tri$ in a disc.

        \item Suppose some triangle of $\tri$
        has two of its edges in $\mathcal{E}$.
        The resulting expansion step ``grows'' both $\mathcal{E}$
        and $\nbd{\mathcal{E}}$ to
        include the entire triangle (see Figure~\ref{fig-expand-face2}).
        This does not change the topology of $\nbd{\mathcal{E}}$,
        which remains a 3-ball possibly with punctures.

        If the third edge of the triangle is internal to $\tri$
        then the intersection $\nbd{\mathcal{E}} \cap \partial \tri$
        does not change.  If the third edge lies in
        the boundary of $\tri$
        then the intersection $\nbd{\mathcal{E}} \cap \partial \tri$
        expands but remains connected (we effectively attach
        a strip surrounding this third edge).

        \begin{figure}[tb]
            \centering
            \subfigure[Growing $\mathcal{E}$ across a
                triangle\label{fig-expand-face2}]%
                {\includegraphics[scale=0.7]{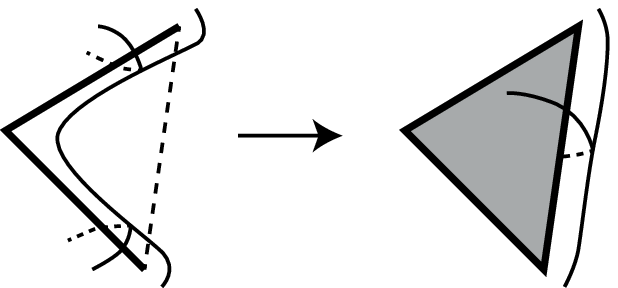}}
            \qquad\qquad
            \subfigure[Expanding $\mathcal{E}$ into a
                triangle\label{fig-expand-face3}]%
                {\includegraphics[scale=0.7]{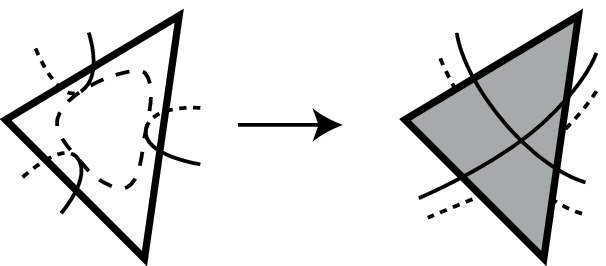}}
            \caption{How expansion affects the neighbourhood
                $\nbd{\mathcal{E}}$}
            \label{fig-expand-face}
        \end{figure}

        \item Suppose some triangle of $\tri$ has all three of its edges
        in $\mathcal{E}$.
        Here the expansion attaches a thickened disc to
        $\nbd{\mathcal{E}}$ (see Figure~\ref{fig-expand-face3}),
        which converts $\nbd{\mathcal{E}}$ from
        a 3-ball with $k$ punctures to a 3-ball with $k+1$ punctures.

        As before, if the triangle is internal to $\tri$
        then the intersection $\nbd{\mathcal{E}} \cap \partial \tri$
        does not change, and
        if the triangle lies in the boundary of $\tri$
        then the intersection expands but remains connected.

        \item Finally, if a tetrahedron has all four of its
        faces in $\mathcal{E}$,
        the resulting expansion simply ``closes off'' one of the punctures
        in $\nbd{\mathcal{E}}$ by filling it with a ball.
    \end{itemize}

    Inducting over the step-by-step construction of $\mathcal{E}$
    now establishes claim~B.

    \textbf{Claim C:}
    \emph{Some component of the link $\lk{\mathcal{E}}$ is a
    disc that is not a vertex link.}

    Let $I = \nbd{\mathcal{E}} \cap \partial \tri$.
    By claim~B, $\nbd{\mathcal{E}}$ is bounded by spheres,
    and so $I$ (which is a non-empty, connected subset of this boundary)
    must be a sphere with zero or more punctures.
    Because $\partial \tri$ is a torus, $I \neq \partial \tri$;
    therefore the surface $I$ has boundary, and must be a sphere with
    at least one puncture.

    Let $S$ be the particular sphere bounding $\nbd{\mathcal{E}}$ that
    contains $I$ as a subset.
    When combined, the link $\lk{\mathcal{E}}$ and $I$ together form the
    entire boundary of $\nbd{\mathcal{E}}$;
    therefore $\lk{\mathcal{E}}$ is the disjoint union of (i)~zero or more
    spheres (excluding $S$), and (ii)~one or more discs (which
    ``plug the punctures'' left by $I$ on the sphere $S$).
    If all of these latter discs are vertex links, then all of the punctures
    in $I$ must also be filled with discs on $\partial \tri$
    (see Figure~\ref{fig-linkball}); however,
    this would make $\partial \tri$ a sphere (not a torus), and so
    at least one of these discs must not be a vertex link.
    This establishes claim~C.

    \begin{figure}[tb]
        \centering
        \includegraphics[scale=0.9]{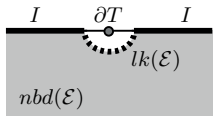}
        \caption{A vertex linking component of $\lk{\mathcal{E}}$}
        \label{fig-linkball}
    \end{figure}

    \textbf{Conclusion:}
    By claims~A and C, some component of $\lk{\mathcal{E}}$ is the
    non-vertex-linking normal disc that we seek.  All that remains is to show
    that we can construct this disc in $O(n^2)$ time.

    To build the subcomplex $\mathcal{E}$ from the original edge $e$
    takes $O(n^2)$ time:
    each expansion step can be tested and performed in $O(n)$
    time, and since $\tri$ contains $O(n)$ edges, triangles and tetrahedra
    there are $O(n)$ expansion steps in total.

    Building the normal surface $\lk{\mathcal{E}}$ now takes $O(n)$ time,
    since (from Claim~A) we simply insert $O(1)$ triangles and/or
    quadrilaterals in each tetrahedron.
    Moreover, this normal surface contains $O(n)$ triangles and
    quadrilaterals in total, and so in $O(n)$ time we can
    extract the connected components of $\lk{\mathcal{E}}$,
    compute their Euler characteristics, and identify one such component
    that is a disc but not a vertex link.
    This concludes the proof of Lemma~\ref{l-onevtx}.
\end{proof}


\subsection{Searching for positive Euler characteristic} \label{s-search}


We now give the details for step~\ref{en-alg-search} of
Algorithm~\ref{alg-unknot},
in which we search through our triangulation for a
connected normal surface which is not the vertex link
and has positive Euler characteristic.\footnote{%
    Recall from the proof of Lemma~\ref{l-crush} that,
    within a knot complement, having positive Euler characteristic encodes the
    fact that a connected surface is either a sphere or a disc.}
This search sits at the heart of the algorithm,
and is its major bottleneck.

\begin{assumptions}
    Throughout this section,
    $\tri$ is a one-vertex, $n$-tetrahedron
    triangulation of a knot complement $\kcomp$.
\end{assumptions}

In principle, our search treats the Euler characteristic
$\chi\co \R^{7n} \to \R$ as a linear
objective function, and maximises $\chi$ over the space of
admissible points in $\R^{7n}$
using branch-and-bound techniques.
If the surface we are seeking exists then the maximum $\chi$ will be
unbounded (because the matching equations, quadrilateral constraints and
$\chi$ are all homogeneous, and so admissible points with $\chi > 0$
can be scaled arbitrarily).
If the surface we are seeking does not exist then
the maximum $\chi$ will be zero (because $\mathbf{0} \in \R^{7n}$ is always
admissible).
In practice, since we only need to distinguish between a zero or
positive maximum, we formulate our search purely in terms of branching and
feasibility tests.

We begin this section with Lemmata~\ref{l-conn} and \ref{l-nonvtxlink},
which set up constraints for our search:
Lemma~\ref{l-conn} shows we can ensure our surface is connected
by setting as many coordinates as possible to zero,
and Lemma~\ref{l-nonvtxlink} shows we can ensure our surface is not the
vertex link by setting at least one triangle coordinate to zero.
We follow with a description of
Algorithm~\ref{alg-search-broad}, which lays out the
structure of the search, and then we discuss the details of the
branching scheme, prove correctness, and analyse the running time.

Recall from the preliminaries section that, if $S$ is a normal surface
in $\tri$,
then $\mathbf{v}(S)$ denotes the vector representation of $S$ in $\R^{7n}$.
Recall also that a point $\mathbf{x} \in \R^{7n}$ is \emph{admissible}
if it satisfies $\mathbf{x} \geq 0$, the matching equations
$A\mathbf{x}=0$, and the quadrilateral constraints.

\begin{lemma} \label{l-conn}
    Let $S$ be a normal surface in $\tri$ with positive Euler
    characteristic, and suppose there is no normal surface $S'$ in $\tri$
    with positive Euler characteristic with the following properties:
    (i)~whenever the $i$th coordinate of $\mathbf{v}(S)$ is zero then the
    $i$th coordinate of $\mathbf{v}(S')$ is likewise zero; and
    (ii)~there is some $i$ for which the $i$th coordinate of $\mathbf{v}(S)$
    is non-zero but the $i$th coordinate of $\mathbf{v}(S')$ is zero.

    Then the smallest positive rational multiple of $\mathbf{v}(S)$ whose
    coordinates are all integers represents a
    connected normal surface with positive Euler characteristic.
\end{lemma}

\begin{proof}
    Let $\mathcal{C}$ denote the polyhedral cone in $\R^{7n}$ defined by
    $\mathbf{x} \geq 0$ and $A\mathbf{x}=0$.  We first show that the vector
    $\mathbf{v}(S)$ lies on an extreme ray of $\mathcal{C}$ (in the
    language of normal surface theory, such an $S$ is called a
    \emph{vertex normal surface}).

    Suppose that $\mathbf{v}(S)$ does not lie on an extreme ray of
    $\mathcal{C}$.
    Since $\mathbf{v}(S)$ is admissible we have
    $\mathbf{v}(S) \in \mathcal{C}$, and so we can express
    $\mathbf{v}(S)$ as a finite non-negative linear combination
    $\mathbf{v}(S) = \lambda_1 \mathbf{e}_1 + \ldots + \lambda_k \mathbf{e}_k$,
    where each $\mathbf{e}_i$ lies on a different extreme ray of $\mathcal{C}$,
    each $\lambda_i > 0$, and there are $k \geq 2$ terms.
    Because $\mathcal{C}$ is a rational cone, we can take each
    $\mathbf{e}_i$ to be an \emph{integer} vector.  Since
    $\mathcal{C}$ lies in the non-negative orthant, it follows that
    whenever the $i$th coordinate of $\mathbf{v}(S)$ is
    zero then the $i$th coordinate of each $\mathbf{e}_i$ must
    likewise be zero.
    In particular, since $\mathbf{v}(S)$
    satisfies the quadrilateral constraints, each
    $\mathbf{e}_i$ also satisfies the quadrilateral constraints,
    and so each $\mathbf{e}_i = \mathbf{v}(S_i)$ for some normal surface $S_i$.

    Recall that the Euler characteristic $\chi$ is a
    homogeneous linear function on $\R^{7n}$.
    Since $\chi(\mathbf{v}(S))>0$,
    we have $\chi(\mathbf{e}_p)>0$ for some summand $\mathbf{e}_p$;
    i.e., the corresponding surface $S_p$ has positive Euler characteristic.
    Take any other summand $\mathbf{e}_q$ where $q \neq p$.
    Since the inequalities that define
    $\mathcal{C}$ are all of the form $x_i \geq 0$, the extreme rays of
    $\mathcal{C}$ are defined by which coordinates are zero and which
    are non-zero; in particular, there must be some coordinate which is
    non-zero in $\mathbf{e}_q$ but zero in $\mathbf{e}_p$.
    This coordinate is therefore non-zero in $\mathbf{v}(S)$ but zero in
    $\mathbf{v}(S_p)$, and so the surface $S_p$
    satisfies all of the properties of $S'$ in the lemma statement,
    yielding a contradiction.

    Therefore $\mathbf{v}(S)$ lies on an extreme ray of $\mathcal{C}$.
    Let $\mathbf{u}$ be the smallest positive rational
    multiple of $\mathbf{v}(S)$ with integer coordinates.  It is clear that,
    like $\mathbf{v}(S)$,
    $\mathbf{u}$ is also admissible with $\chi(\mathbf{u})>0$;
    that is, $\mathbf{u}$ represents some normal surface $U$ with
    positive Euler characteristic.
    All that remains to show is that $U$ is connected.

    If $U$ is not connected, we can express it as the disjoint union of
    non-empty normal surfaces $U = X \cup Y$, whereupon
    $\mathbf{u}=\mathbf{v}(X)+\mathbf{v}(Y)$.  Since $\mathbf{u}$ lies
    on an extreme ray of $\mathcal{C}$, it must be that both $\mathbf{v}(X)$
    and $\mathbf{v}(Y)$ are \emph{smaller} integer multiples of
    $\mathbf{u}$, in contradiction with the definition of $\mathbf{u}$.
\end{proof}

\begin{lemma} \label{l-nonvtxlink}
    A connected normal surface $S$ in $\tri$ is the vertex link if and
    only if the vector $\mathbf{v}(S) \in \R^{7n}$ has
    none of its $4n$ triangle coordinates equal to zero.
\end{lemma}

\begin{proof}
    This is an immediate consequence of the fact that $\tri$ is
    one-vertex.  The vertex link itself has all triangle coordinates
    equal to one, and any other surface whose triangle coordinates are all
    positive must be the disconnected union of the vertex link with one or
    more other surfaces.
\end{proof}

\begin{algorithm} \label{alg-search-broad}
    To find a connected normal surface $S$ which is not the vertex link
    and which has positive Euler characteristic:
    \begin{enumerate}
        \item \label{en-broad-adm}
        Search for an admissible point
        $\mathbf{p} \in \R^{7n}$ for which $\chi(\mathbf{p}) \geq 1$,
        and for which at least one of the $4n$ triangle coordinates is zero.
        If no such $\mathbf{p}$ exists then the surface $S$ does not
        exist either.

        \item \label{en-broad-shrink}
        Construct a system $\mathcal{L}$ of linear constraints,
        initially defined by $\mathbf{x} \geq 0$,
        $\chi(\mathbf{x}) \geq 1$, and the matching equations
        $A\mathbf{x} = 0$.  Using the point $\mathbf{p}$ located in the
        previous step:
        \begin{enumerate}
            \item For each coordinate $i=1,\ldots,7n$ with $p_i=0$,
            add the additional constraint $x_i=0$ to the system $\mathcal{L}$.

            \item Then, for each coordinate $i=1,\ldots,7n$
            with $p_i \neq 0$,
            add the constraint $x_i=0$ to $\mathcal{L}$
            and test whether $\mathcal{L}$ is feasible
            (i.e., has any solutions).  If so, keep the
            constraint $x_i=0$ in $\mathcal{L}$; if not, remove it again.
        \end{enumerate}

        \item \label{en-broad-int}
        Let $\mathbf{q}$ be a solution to the final system $\mathcal{L}$,
        and let $\lambda\mathbf{q}$ be the smallest positive rational multiple
        of $\mathbf{q}$ whose coordinates are all integers.  Then
        $\lambda\mathbf{q}$ is an integer vector in $\R^{7n}$
        that represents the desired surface $S$.
    \end{enumerate}
\end{algorithm}

In summary: step~\ref{en-broad-adm} contains the bulk of the work in
locating a solution if one exists, and requires worst-case exponential time.
If a solution \emph{is} found then
steps~\ref{en-broad-shrink} and~\ref{en-broad-int} refine it in
polynomial time by setting as
many additional coordinates to zero as possible, which ensures that
the final normal surface is connected (as in Lemma~\ref{l-conn})
without violating the quadrilateral constraints.

We analyse this algorithm shortly (see Lemma~\ref{l-search-broad},
which proves correctness and bounds the time complexity).
First, however, we describe in detail
the branching scheme used to search for the point $\mathbf{p}$ in
step~\ref{en-broad-adm}.
Specifically, we are searching for a (rational) point
$\mathbf{x} \in \R^{7n}$ that satisfies:
\begin{enumerate}[(i)]
\item $\mathbf{x} \geq 0$;
\item the matching equations $A\mathbf{x}=0$;
\item the quadrilateral constraints;
\item $\chi(\mathbf{x}) \geq 1$; and
\item at least one of the $4n$ triangle coordinates of $\mathbf{x}$ is
equal to zero.
\end{enumerate}
For convenience, let $x^\triangle_i$ denote the $i$th
triangle coordinate of $\mathbf{x}$ ($1 \leq i \leq 4n$), and let
$x^\square_{t,k}$ denote the quadrilateral coordinate
that counts the $k$th type of quadrilateral in the $t$th tetrahedron
($1 \leq t \leq n$, $1 \leq k \leq 3$).

Conditions (i), (ii) and (iv) are just linear constraints over
$\R^{7n}$, and can be used seamlessly with linear programming.
The remaining conditions (iii) and (v) are combinatorial,
and cause more difficulties.
We could formulate them as integer constraints,
but the resulting integer programming problems are impractical for off-the-shelf
solvers (they induce extremely large coefficients, and require
exact arithmetic) \cite{burton12-crosscap}.
Here we enforce these combinatorial constraints using branching:
\begin{itemize}
    \item \emph{Triangle branching:} \\
    We must make an initial decision on which triangle coordinate will be
    zero.
    To avoid redundancy, we select the \emph{first} such
    coordinate; that is, we choose some $i \in \{1,\ldots,4n\}$
    for which $x^\triangle_1,\ldots,x^\triangle_{i-1} \geq 1$ and
    $x^\triangle_i = 0$.
    This gives $4n$ branches, one for each $i$.

    \item \emph{Quadrilateral branching:} \\
    For each $t=1,\ldots,n$, we must decide which quadrilateral
    coordinate in the $t$th tetrahedron will be non-zero, if any.
    That is, we choose between branches
    (a)~$x^\square_{t,1}=x^\square_{t,2}=x^\square_{t,3}=0$;
    (b)~$x^\square_{t,1} \geq 1$ and $x^\square_{t,2}=x^\square_{t,3}=0$;
    (c)~$x^\square_{t,2} \geq 1$ and $x^\square_{t,1}=x^\square_{t,3}=0$;
    or
    (d)~$x^\square_{t,3} \geq 1$ and $x^\square_{t,1}=x^\square_{t,2}=0$.
\end{itemize}

Here we replace all strict inequalities $x_i > 0$ with
non-strict inequalities $x_i \geq 1$: this
makes the linear programming simpler, and does not affect the existence
of solutions because any solution $\mathbf{x}$ can be rescaled to some
solution $\lambda\mathbf{x}$ with all coordinates $\geq 1$.

\begin{figure}[tb]
    \centering
    \includegraphics[scale=0.75]{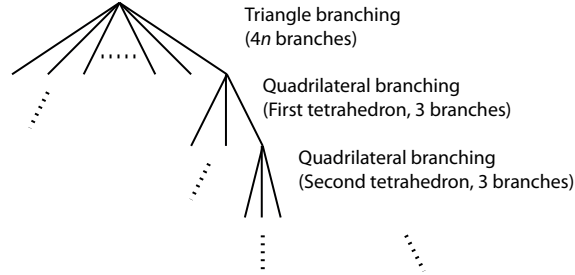}
    \caption{The combinatorial search tree}
    \label{fig-tree}
\end{figure}

We arrange these branches into a combinatorial search tree, as
illustrated in Figure~\ref{fig-tree}: the triangle branch is chosen
first, and then the $n$ quadrilateral branches are chosen in some order
(which we describe shortly).  To reduce the total amount of branching,
we merge quadrilateral branches (a) and (b) into the single branch
(a*)~$x^\square_{t,1} \geq 0$ and $x^\square_{t,2}=x^\square_{t,3}=0$,
so there are only three (not four) branches at each quadrilateral decision.

To run step~\ref{en-broad-adm} of Algorithm~\ref{alg-search-broad}, we begin
at the root of this search tree and traverse it in a depth-first
fashion.  At all stages we maintain a system $\mathcal{N}$ of
linear constraints corresponding to the branches that have been chosen:
\begin{itemize}
    \item At the root of the tree we initialise $\mathcal{N}$ to the
    constraints $\mathbf{x} \geq 0$, the matching equations
    $A\mathbf{x}=0$, and $\chi(\mathbf{x}) \geq 1$.

    \item Each time we follow a new branch down, we extend $\mathcal{N}$
    by adding additional constraints to describe the branch that was
    chosen.  For instance, when following a triangle branch we would add
    constraints of the form
    $x^\triangle_1,\ldots,x^\triangle_{i-1} \geq 1$ and $x^\triangle_i = 0$,
    and when following quadrilateral branch~(d) for some tetrahedron
    we would add the constraints
    $x^\square_{t,3} \geq 1$ and $x^\square_{t,1}=x^\square_{t,2}=0$.

    \item Conversely, each time we backtrack and follow a branch up,
    we remove these additional constraints from $\mathcal{N}$.
\end{itemize}

Crucially, each time we enter a node in the search tree we
\emph{test the system $\mathcal{N}$ for feasibility} (i.e., whether it
has a solution).  If the system is not feasible, then we backtrack
immediately (since the branches chosen thus far are inconsistent,
and cannot lead to a desired admissible point $\mathbf{p} \in \R^{7n}$).
If we ever reach a \emph{leaf node} of the tree---that is, a node at
which we have followed a triangle branch plus quadrilateral branches
for all $n$ tetrahedra---then
any solution to the system $\mathcal{N}$ is the admissible point
$\mathbf{p} \in \R^{7n}$ that we require, at which point we can
terminate the search and move immediately to
step~\ref{en-broad-shrink} of Algorithm~\ref{alg-search-broad}.

We emphasise the underlying branch-and-bound motivations behind this search.
If we treat $\chi$ as a linear objective function to maximise,
then our feasibility tests at each node play the role of
\emph{linear relaxations}---in the absence of constraints for
decisions not yet made, they test whether
the maximum $\chi$ could possibly be positive in the
subtree below the current node.

As is common in branch-and-bound, the order in which we make
individual branching decisions can have an enormous impact on the
overall running time.
With this in mind, we now describe important enhancements to the
vertical layout of the search tree, i.e., the order in which we
make our $n$ quadrilateral decisions.
Figure~\ref{fig-tree} depicts these decisions as
ordered by tetrahedron number; however, we can dynamically reorder these
decisions to take full advantage of the interaction between different
constraints.
Experimentation suggests that the following optimisations are crucial to
the polynomial-time behaviour that we observe in practice
in Section~\ref{s-expt}.

\begin{itemize}
    \item After choosing an initial triangle branch,
    if the triangle coordinate that we set to zero was in the $i$th
    tetrahedron, then we immediately branch on quadrilateral coordinates in the
    \emph{same ($i$th) tetrahedron}.

    This is because setting a triangle \emph{and} two quadrilateral
    coordinates to zero in the same tetrahedron is an extremely powerful
    constraint, and immediately eliminates several other ``nearby''
    quadrilateral and triangle coordinates from surrounding tetrahedra.

    \item For subsequent quadrilateral decisions, we greedily branch on the
    tetrahedron that yields the \emph{fewest feasible child nodes}.
    That is, for each unused tetrahedron we build the three
    constraint systems $\mathcal{N}$ that would result at each
    of the three child nodes if we branched next on that tetrahedron,
    and we count how many of these three systems are feasible.
    We then branch on quadrilateral coordinates in whichever
    tetrahedron minimises this count.

    Essentially, this greedy approach allows us to lock in
    additional ``forced'' constraints as quickly as possible.
    Moreover, if there is some tetrahedron in which
    \emph{none} of the three quadrilateral branches yield a feasible
    child node, our approach detects this and allows us to
    backtrack immediately.
    A drawback of this greedy method is that it requires a linear number of
    feasibility tests at each node of the search tree, but this only
    affects the running time by a polynomial factor.
\end{itemize}

\begin{lemma} \label{l-search-broad}
    Algorithm~\ref{alg-search-broad}
    (which implements the search in step~\ref{en-alg-search} of
    Algorithm~\ref{alg-unknot})
    is correct in either finding a
    connected normal surface $S$ which is not the vertex link and
    has positive Euler characteristic, or in
    showing that no such surface exists.
    Moreover, the algorithm runs in time $O(3^n \times \mathrm{poly}(n))$.
\end{lemma}

\begin{proof}
    Recall that step~\ref{en-broad-adm} of Algorithm~\ref{alg-search-broad}
    searches for an admissible point $\mathbf{p}\in\R^{7n}$ for which
    $\chi(\mathbf{p})\geq 1$, and for which at least one of the $4n$
    triangle coordinates is zero.

    Suppose we fail to find such a point $\mathbf{p}$ in
    step~\ref{en-broad-adm}.
    If there \emph{were} some connected normal surface $S$ which is not the
    vertex link and has positive Euler characteristic, then
    (by Lemma~\ref{l-nonvtxlink}) the
    vector $\mathbf{p} = \mathbf{v}(S) \in \R^{7n}$ would satisfy our
    search criteria.  Therefore the algorithm is correct in concluding
    that no such surface $S$ exists.

    Suppose we do find such a point $\mathbf{p}$ in step~\ref{en-broad-adm}.
    Consider now the point $\mathbf{q}$ that
    we construct in steps~\ref{en-broad-shrink} and \ref{en-broad-int}
    of the algorithm.  It is clear that $\mathbf{q}$ exists (i.e.,
    the final system $\mathcal{L}$ is feasible), since the
    initial system $\mathcal{L}$ has $\mathbf{p}$ as a solution,
    each additional constraint added in step~\ref{en-broad-shrink}(a)
    is again satisfied by $\mathbf{p}$, and
    each additional constraint added in step~\ref{en-broad-shrink}(b)
    is explicitly tested for feasibility.

    We now claim that $\mathbf{q}$ is admissible.  This is true because
    the conditions
    $\mathbf{q}\geq 0$ and $A\mathbf{q}=0$ are
    built into the system $\mathcal{L}$; moreover, by
    step~\ref{en-broad-shrink}(a) we know that each coordinate that is
    zero in $\mathbf{p}$ is also zero in $\mathbf{q}$, and so because
    $\mathbf{p}$ satisfies the quadrilateral constraints then
    $\mathbf{q}$ must also.

    We next show that $\mathbf{q}$ scales to a rational vector.
    Let $\mathcal{P}$ be the rational polyhedron in $\R^{7n}$ defined by
    the system $\mathcal{L}$.  From the structure of the inequalities
    that define $\mathcal{L}$, we see that each facet of $\mathcal{P}$
    is obtained by intersecting $\mathcal{P}$ with a supporting
    hyperplane of the form (i)~$x_i=0$, or (ii)~$\chi(\mathbf{x}) = 1$.
    From step~\ref{en-broad-shrink} we see that any intersection
    of $\mathcal{P}$ with a hyperplane $x_i=0$ either includes all of
    $\mathcal{P}$ or is the empty set; either way we cannot obtain a
    facet.  Therefore $\mathcal{P}$ has only one facet (the intersection
    with $\chi(\mathbf{x})=1$), and it follows that $\mathcal{P}$ is a
    one-dimensional ray.  Since $\mathcal{L}$ is a rational system we
    conclude that the solution $\mathbf{q}$ scales down to a rational
    point (the vertex of $\mathcal{P}$ at the beginning of this ray).

    It follows that in step~\ref{en-broad-int} the multiple
    $\lambda\mathbf{q}$ is well-defined, and represents an admissible
    integer vector with $\chi(\lambda\mathbf{q}) > 0$.
    We therefore have $\lambda\mathbf{q}=\mathbf{v}(S)$ for some normal
    surface $S$ in $\tri$ with positive Euler characteristic.
    Because some triangle coordinate of $\mathbf{p}$ is
    zero, step~\ref{en-broad-shrink}(a) ensures that some triangle coordinate of
    $\mathbf{v}(S)$ is zero, and it follows from
    Lemma~\ref{l-nonvtxlink} that $S$ is not the vertex link.

    We now use Lemma~\ref{l-conn} to show that this surface is connected.
    Suppose there were some normal surface $S'$ with positive
    Euler characteristic where (i)~for every coordinate of
    $\lambda\mathbf{q}=\mathbf{v}(S)$ which is zero, the corresponding
    coordinate of $\mathbf{v}(S')$ is likewise zero; and
    (ii)~there is some $i$ for which the
    $i$th coordinate of $\lambda\mathbf{q}=\mathbf{v}(S)$ is non-zero
    but the $i$th coordinate of $\mathbf{v}(S')$ is zero.
    Then in step~\ref{en-broad-shrink}(b) of the algorithm, for this
    particular value of $i$, the constraint $x_i=0$ would have given a feasible
    system (having $\mathbf{v}(S')$ as a solution).
    Therefore $x_i=0$ would have been a condition in
    the final system $\mathcal{L}$,
    and $\mathbf{q}$ could not have been a solution---a contradiction.
    Therefore $S$ satisfies the conditions of Lemma~\ref{l-conn},
    and so (because our scaling factor $\lambda$ realises the smallest
    possible integer multiple) the surface $S$ must be connected.

    We now have that $S$ is a connected normal surface which is
    not the vertex link and which has positive Euler characteristic.
    This concludes the proof that the output of
    Algorithm~\ref{alg-search-broad} is correct.

    To finish, we analyse the time complexity of the algorithm.
    First we observe that every feasibility test that appears in the algorithm
    involves a linear number of constraints, each with integer coefficients of
    size $O(1)$.  In particular:
    \begin{itemize}
        \item The system of matching equations $A\mathbf{x}=0$ contains
        at most $6n$ equations, each of the form
        $x^\triangle_i + x^\square_{j,k} =
        x^\triangle_u + x^\square_{v,w}$.

        \item Recall from the preliminaries section that there are many
        choices for our linear Euler characteristic function $\chi$.
        To ensure that this function has
        constant sized integer coefficients, we choose the
        formulation $\chi(\mathbf{x})=\chi_2(\mathbf{x})-\chi_1(\mathbf{x})+
        \chi_0(\mathbf{x})$,
        where:
        \begin{itemize}
            \item $\chi_2(\mathbf{x})$ is the sum of all coordinates of
            $\mathbf{x}$;
            \item $\chi_1(\mathbf{x}) = \sum_F \chi_1^F(\mathbf{x})$,
            where $F$ ranges over all triangular faces $F$ of the
            triangulation $\tri$, and
            where each $\chi_1^F(\mathbf{x})$ is computed by choosing
            an arbitrary tetrahedron $\Delta_F$ that contains $F$,
            and summing the six
            coordinates of $\mathbf{x}$ that correspond to normal discs
            in $\Delta_F$ that meet $F$;
            \item $\chi_0(\mathbf{x}) = \sum_e \chi_0^e(\mathbf{x})$,
            where $e$ ranges over all edges of $\tri$,
            and each $\chi_0^e(\mathbf{x})$ is likewise computed by
            choosing an arbitrary tetrahedron $\Delta_e$ that contains $e$
            and summing the four
            coordinates of $\mathbf{x}$ that correspond to normal discs
            in $\Delta_e$ that meet $e$.
        \end{itemize}
        If some face $F$ appears multiple times within the
        tetrahedron $\Delta_F$ (i.e., two faces of $\Delta_F$ are
        identified together), then we restrict our
        attention to just one of these appearances when computing
        $\chi_1^F(\mathbf{x})$; likewise with $\chi_0^e(\mathbf{x})$.

        We note that this choice of Euler characteristic function is a
        valid one:
        if the vector $\mathbf{x}$ represents a normal surface $S$ then
        $\chi_2(\mathbf{x})$, $\chi_1(\mathbf{x})$ and $\chi_0(\mathbf{x})$
        count the number of discs, edges and vertices respectively in
        $S$, and so $\chi(\mathbf{x})$ is indeed
        the Euler characteristic of $S$.

        Regarding the coefficients of this function:
        since each coordinate of $\mathbf{x}$ may appear in at most four
        distinct terms $\chi_1^F(\mathbf{x})$ and at most four distinct
        terms $\chi_0^e(\mathbf{x})$, the formulation above expresses
        $\chi(\mathbf{x})$ as a linear function with integer
        coefficients all in the range $-3,\ldots,+5$.
    \end{itemize}

    This establishes that every feasibility test that appears in the algorithm
    involves a linear number of constraints each with constant sized
    coefficients.  It follows that every such test can be solved in
    polynomial time using linear programming techniques.

    We can now measure the time complexity of each step of the algorithm:
    \begin{itemize}
        \item
        For step~\ref{en-broad-adm}
        we simply count nodes: there are $4n$ branches for
        the initial triangle decision and three branches for each of the $n$
        quadrilateral decisions, giving $4n \cdot 3^n$ leaf nodes in the
        worst case.
        Each feasibility test can be run in polynomial time as noted above,
        yielding an overall running time of $O(3^n \times \mathrm{poly}(n))$.

        \item
        Step~\ref{en-broad-shrink} runs in polynomial time because
        throughout its evolution the system $\mathcal{L}$ always contains
        $O(n)$ constraints, we only modify it $O(n)$ times, and for each
        modification we run a single feasibility test which can be
        performed in polynomial time as before.

        \item
        For step~\ref{en-broad-int} we must show that we can
        perform the necessary arithmetic on $\mathbf{q}$ in polynomial time.
        For this it suffices to show that the coordinates of our
        smallest integer multiple $\lambda\mathbf{q}$ each have
        just $O(n)$ bits.

        By following the same argument as used in Lemma~\ref{l-conn},
        we find that $\mathbf{q}$ lies on an extreme ray of the
        polyhedral cone $\mathcal{C}$
        defined by $\mathbf{x} \geq 0$ and $A\mathbf{x}=0$.
        This can also be seen directly:
        by removing the non-homogeneous constraint $\chi(\mathbf{x})\geq 1$
        from the final system $\mathcal{L}$ we extend our
        one-dimensional solution set to a homogeneous ray in $\R^{7n}$,
        and by removing the additional constraints $x_i=0$ that were added in
        step~\ref{en-broad-shrink} we further enlarge this to
        the full cone $\mathcal{C}$, with our original solution set now on
        an extreme ray of $\mathcal{C}$.

        We finish by invoking a result of
        Hass et~al.\ \cite[Lemma~6.1]{hass99-knotnp},
        which shows that any extreme ray of $\mathcal{C}$ can be
        expressed as an integer vector with all
        coordinates bounded by $\exp(O(n))$---that is, by $O(n)$ bits.
    \end{itemize}

    This establishes an overall running time of
    $O(3^n \times \mathrm{poly}(n))$ for Algorithm~\ref{alg-search-broad}.
\end{proof}

We pause to make some final observations on step~\ref{en-broad-adm}
of Algorithm~\ref{alg-search-broad}, where we find the point
$\mathbf{p} \in \R^{7n}$ using our branching scheme.
\begin{itemize}
    \item The quadrilateral constraints are special forms of
    SOS~1 constraints: an \emph{SOS~1}, or
    \emph{special ordered set of type~1}, is a set of variables
    at most one of which may be non-zero \cite{beale70-special}.
    Thus, in effect,
    step~\ref{en-broad-adm} of Algorithm~\ref{alg-search-broad}
    solves a family of
    SOS~1 constrained problems using customised branching rules.

    \item The linear programs that we solve in
    step~\ref{en-broad-adm} of Algorithm~\ref{alg-search-broad}
    do not involve any explicit objective function---we are simply
    interested in testing feasibility.
    One could also introduce an explicit objective function in the
    hope that, if the point $\mathbf{p}$ does exist, then we might find
    it more quickly.  For instance, we could minimise
    the sum of triangle coordinates (in the hope that one of the $4n$
    triangle coordinates might be zero), or minimise the sum of
    quadrilateral coordinates (in the hope that the quadrilateral
    constraints might be satisfied).

    However, since such heuristics
    do not guarantee to satisfy the relevant combinatorial constraints,
    they are primarily useful only in the cases where $\mathbf{p}$
    does exist---that is, where the input knot is trivial and/or the
    triangulation of the knot complement is ``inefficient'' (i.e., it can
    be simplified by crushing).
    We note that such inputs are often already ``easy'',
    in the sense that they can
    typically be resolved (or at least reduced) using fast local
    simplification techniques instead, as mentioned in the introduction.
\end{itemize}


\subsection{Proofs of correctness and running time} \label{s-alg-proofs}

Our final task in this section is to tie everything together:
we prove that the full unknot recognition algorithm is correct
(Theorem~\ref{t-correct}) and describe its worst-case time complexity
(Theorem~\ref{t-fast}).

\begin{theorem}[Correctness] \label{t-correct}
    Algorithm~\ref{alg-unknot} is correct in determining whether
    the input knot $K$ is trivial or non-trivial.
\end{theorem}

\begin{proof}
    All that remains is to show that the claims made in
    step~\ref{en-alg-search} of the algorithm
    (after searching for the normal surface $S$) are correct.

    By Theorem~\ref{t-haken}, it is clear that
    if no surface $S$ is found then $K$ is non-trivial,
    and if $S$ is a disc with non-trivial boundary then $K$ is trivial.
    Otherwise the algorithm tells us to crush $S$,
    whereupon Lemma~\ref{l-crush} shows that the new
    triangulation $\tri'$ represents the same knot complement $\kcomp$
    with strictly fewer than $n$ tetrahedra, as claimed.
\end{proof}

\begin{theorem}[Running time] \label{t-fast}
    Let $c$ be the number of
    crossings in the input knot diagram for Algorithm~\ref{alg-unknot}.
    With the exception of the search in step~\ref{en-alg-search}
    (where we search for the normal surface $S$), every step of
    Algorithm~\ref{alg-unknot} runs in time polynomial in $c$.
    In contrast, the search in step~\ref{en-alg-search} runs in time
    $O(3^n \times \mathrm{poly}(n))$, where $n \in O(c)$ is the number of
    tetrahedra in the triangulation $\tri$.
    Every step (including this search) is repeated at most
    $O(c)$ times.
\end{theorem}

\begin{proof}
    Following Hass et~al.\ \cite[Lemma~7.2]{hass99-knotnp},
    step~\ref{en-alg-tri} of the algorithm
    takes $O(c \log c)$ time and builds a
    triangulation of $\kcomp$ with $n \in O(c)$ tetrahedra.
    By Lemma~\ref{l-onevtx}, the conversion in step~\ref{en-alg-simplify}
    to a one-vertex triangulation then takes $O(c^3)$ time.

    Lemma~\ref{l-search-broad} shows that the search in
    step~\ref{en-alg-search} runs in time $O(3^n \times \mathrm{poly}(n))$.
    Following this search, the only significant actions
    are (i)~testing whether the boundary of $S$ is non-trivial
    in $\partial \tri$,
    and (ii)~potentially crushing the surface $S$.

    Because $\tri$ is a one-vertex triangulation,
    the boundary torus $\partial \tri$ contains just two triangles,
    and the boundary of $S$ follows a non-trivial curve on $\partial \tri$
    if and only if it
    is not a trivial loop encircling the (unique) vertex.
    Therefore we can test $S$ for non-trivial boundary in $O(n)$ time
    just by examining the coordinates of $\mathbf{v}(S)$.
    Finally, Lemma~\ref{l-crush} shows that we can crush the surface
    $S$ in $O(n)$ time also.

    Each iteration through these steps results in either termination
    or a reduction in the number of tetrahedra, and so
    we repeat these steps at most $n \in O(c)$ times.
\end{proof}


\section{Experimental performance} \label{s-expt}

Here we describe the results of extensive testing of the new unknot
recognition algorithm over all $2977$ prime knots with $\leq 12$
crossings.

The algorithm has been implemented in {\cpp} using the open-source
computational topology software package {\regina} \cite{burton04-regina,regina},
and is now built directly into {\regina} as of version~4.94.
We briefly describe some aspects of the implementation, and then
describe the experimental data and results.

\subsection{Implementation}

Although it is possible to test the feasibility of a system of
linear constraints in polynomial time (for instance, using interior
point methods \cite{hacijan79-polylp,karmarkar84-new}),
we use a variant of the simplex
method due to its ease of implementation and its excellent performance
in practical settings (as discussed further in Section~\ref{s-disc}).
Specifically, we use the revised dual simplex method
\cite{lemke54-dualsimplex} with an implementation that exploits the
sparseness of the matching equations $A\mathbf{x}=0$.

For a pivoting rule, we use Dantzig's classical method of choosing the
exiting variable with largest magnitude negative value in the tableaux
\cite{dantzig63-linprog}.
Although this method is fast and works well to reduce the total number
of pivots, it can lead to cycling.  We therefore use Brent's algorithm
to detect cycling \cite{brent80-montecarlo}, and when it occurs
we switch to Bland's rule instead
\cite{bland77-pivoting}, which exhibits weaker performance but does not
cycle.  We note that cycling was detected for some experimental inputs,
and so these cycle-breaking techniques are indeed necessary in practice.

All computations use exact integer and rational arithmetic, provided
by the {GNU} multiple precision arithmetic library \cite{gmp}.
To limit the overhead, we work in native integers wherever possible but
test for overflow on all arithmetical operations, and only switch to
exact arithmetic when necessary.  This behaviour, which improves
performance surprisingly well, is inspired by
(but far less sophisticated than) the lazy evaluation methods
used for exact arithmetic in the \emph{CGAL} computational geometry
library \cite{cgal,bronnimann01-interval}.

\subsection{Experimental results} \label{s-expt-results}

Our experimental data set consists of all
$2977$ prime knots with $\leq 12$ crossings, as taken
from the {\knotinfo} database \cite{www-knotinfo-jun11}.
This is intended as an exhaustive and ``punishing'' data set, where
simplification tools cannot solve the problems (since the inputs are
non-trivial knots), and where our branching search
needs to conclusively determine that certain normal surfaces
do \emph{not} exist (so there is no chance for early termination).

As is common in computational topology, our first step before
running any other algorithms is to \emph{simplify} the input
triangulations using a suite of local moves; we use the ready-made
(and polynomial time) suite from {\regina} \cite{burton13-regina}.

The resulting triangulations are large, with up to $n=50$ tetrahedra.
It is a testament to the strength of the simplification suite that
for all $2977$ inputs we only require a single pass through
Algorithm~\ref{alg-search-broad} (the branching search)---we never need
to crush away ``junk'' discs or spheres and run the search again
(as in step~\ref{en-alg-search} of Algorithm~\ref{alg-unknot}).
Similarly, we find that in practice the simplification suite
always produces a one-vertex triangulation immediately, with no need
for the complex operations described in Lemma~\ref{l-onevtx}.


\begin{figure}[tb]
    \centering
    \includegraphics[scale=0.4]{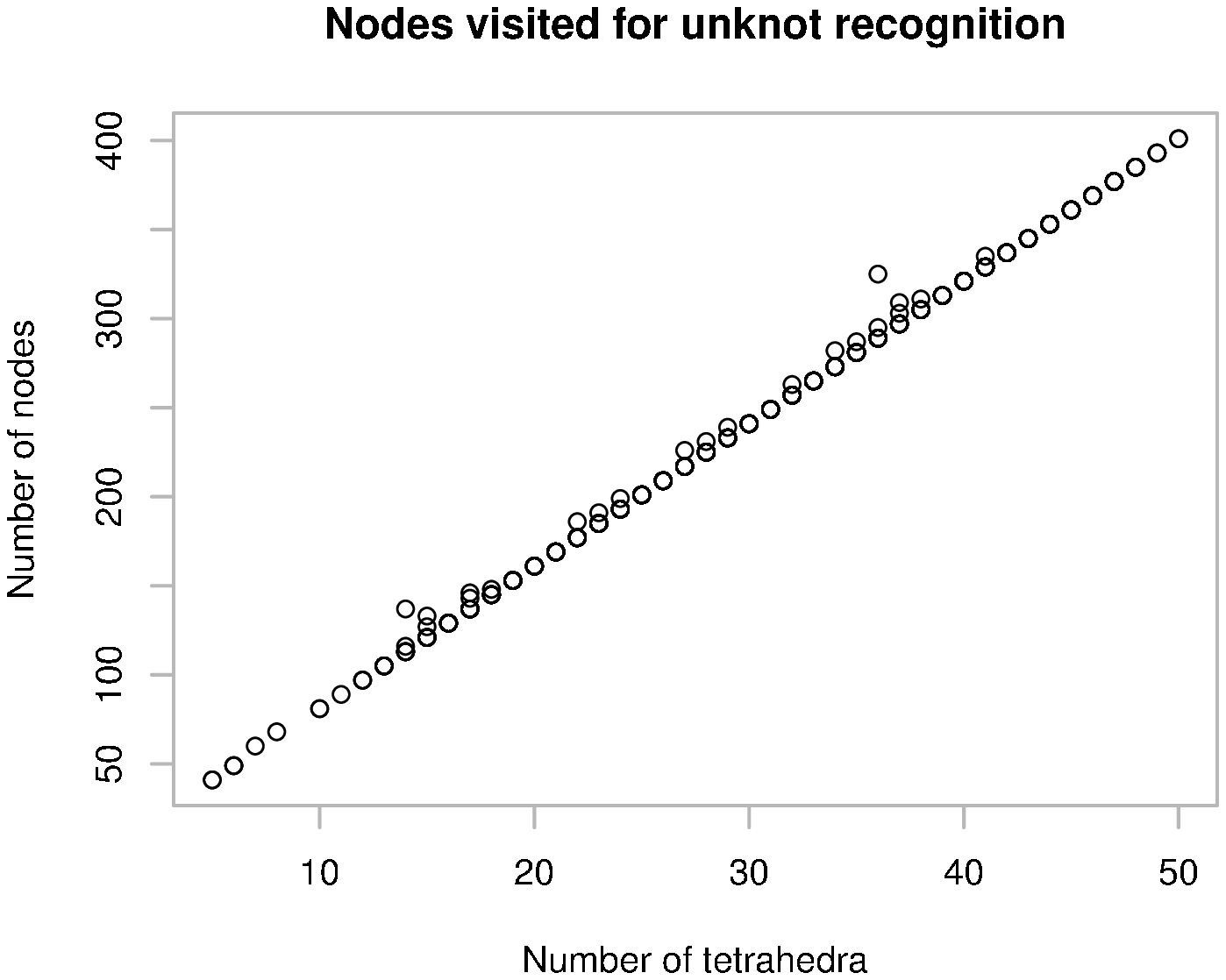} \qquad
    \includegraphics[scale=0.4]{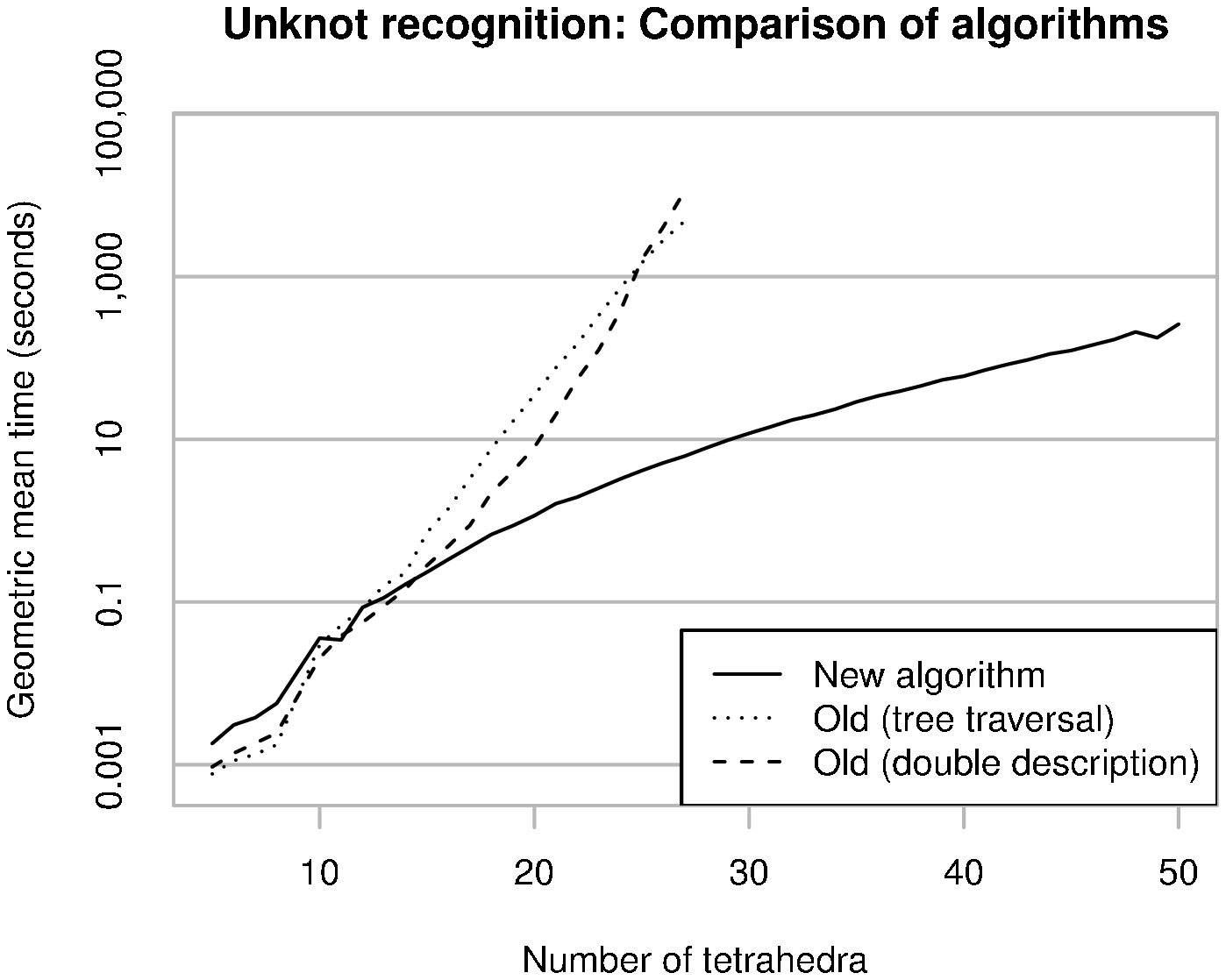}
    \caption{Performance summary for unknot recognition}
    \label{fig-knotnodes}
\end{figure}

The first plot of Figure~\ref{fig-knotnodes} counts how many nodes the
new algorithm visits in the branching
search tree (as described in Section~\ref{s-search});\footnote{%
    What ``visit'' means in this context relies on
    details of the implementation, but up to asymptotics this is
    not important.}
this measure is crucial because it determines how many
linear programming problems we solve,
and is the source of the worst-case exponential running time.
The results are unequivocally linear:
in every case the number of nodes lies between $8n$ and $10n$, and
for all but 24 of the inputs the figure is precisely $8n+1$
(the smallest possible, indicating that we never need to
branch on quadrilateral coordinates at all).\footnote{%
    The exact figure of $8n+1$ is an artefact of the
    implementation: the triangle branching is implemented not as
    a single branch with $4n$ options, but using a binary branching tree
    with $4n+1$ leaves.}
This linear growth in the number of nodes corresponds to
a polynomial running time, and explains the exceptional performance
of the algorithm.

The second plot of Figure~\ref{fig-knotnodes} measures running times,
comparing the new algorithm
against prior state-of-the-art algorithms that rely on an
\emph{enumeration} of candidate normal discs.
These prior algorithms use two different enumeration techniques:
one based on the
double description method \cite{burton10-dd}, and one based on a more
recent tree traversal method \cite{burton13-tree}.
These prior algorithms are
also implemented in {\regina}, and their code is heavily optimised; see
\cite{burton13-regina} for an overview of each.

\begin{figure}[tb]
    \centering
    \includegraphics[scale=0.4]{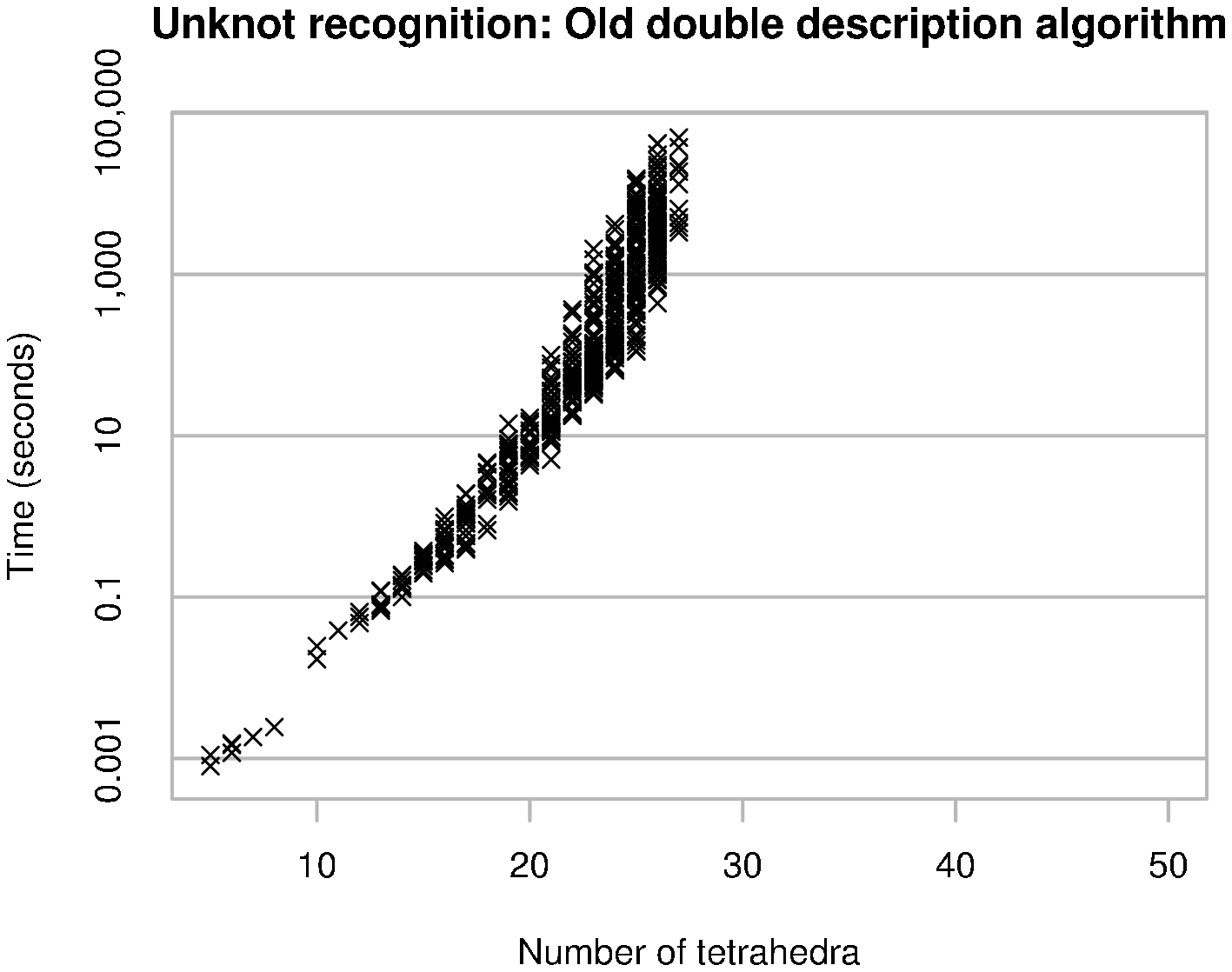} \qquad
    \includegraphics[scale=0.4]{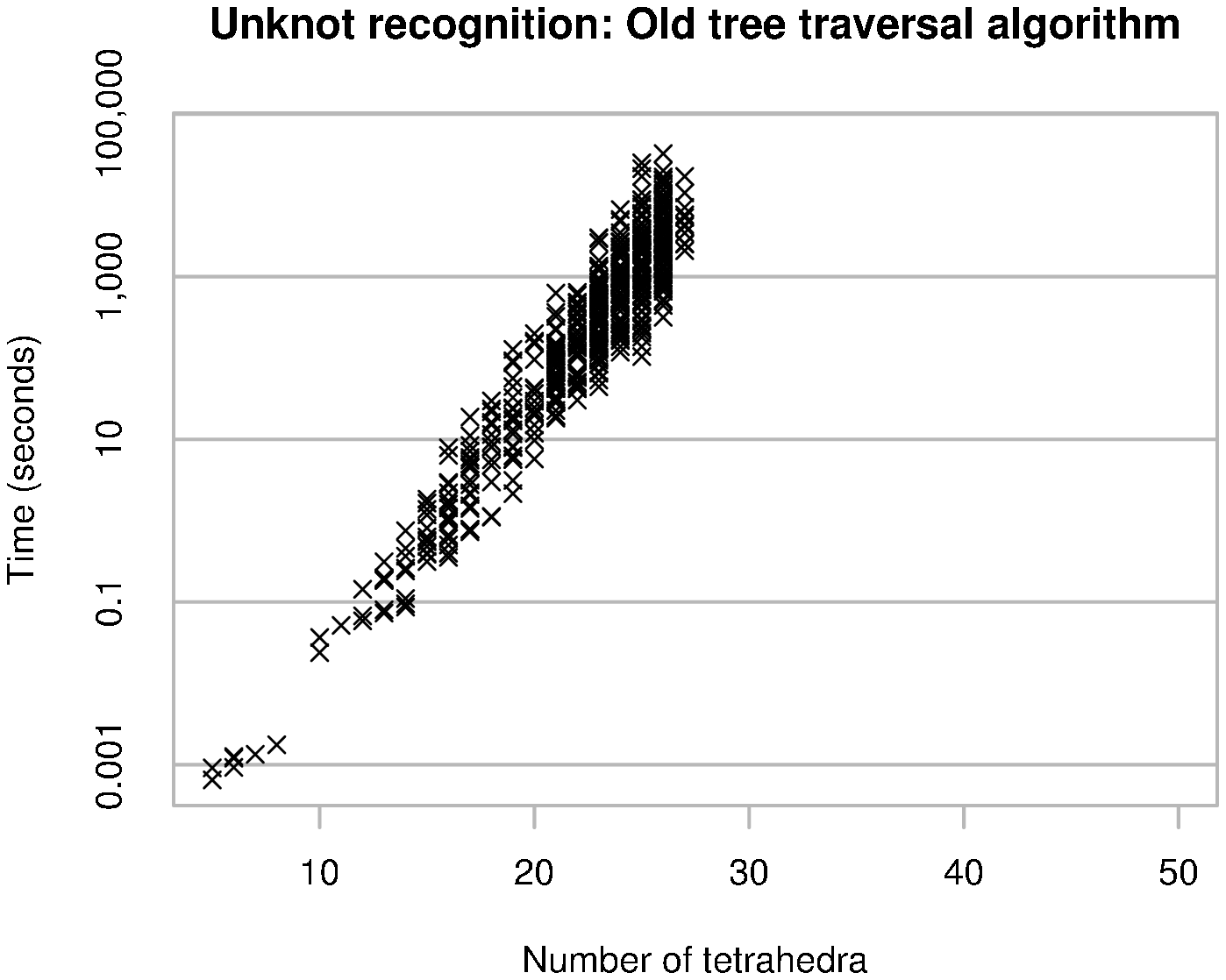} \bigskip \\
    \includegraphics[scale=0.4]{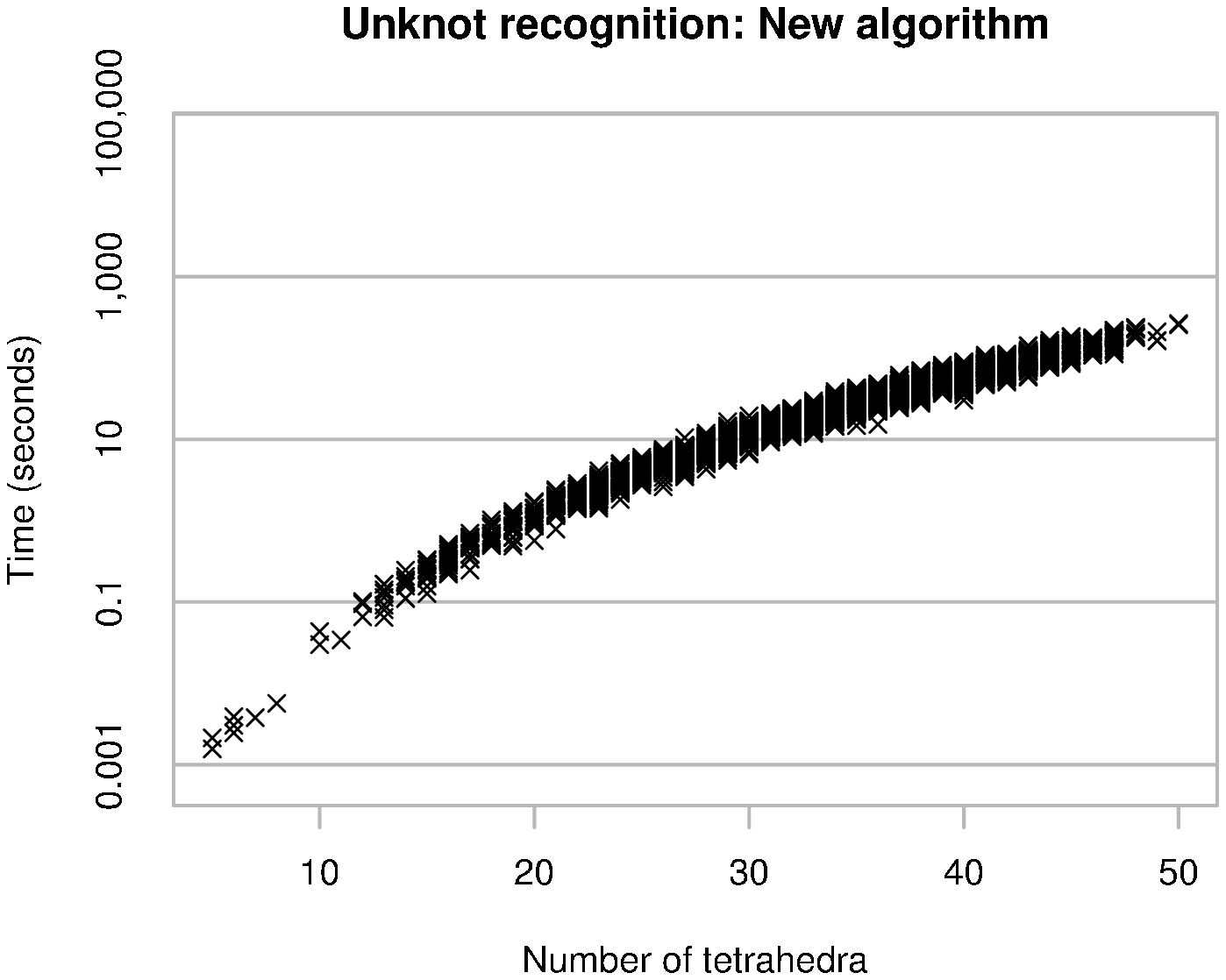}
    \caption{Detailed running times for unknot recognition
        on the first 2977 prime knots}
    \label{fig-knotperf}
\end{figure}

All running times are single-threaded
(measured on a 2.93~GHz Intel Core~i7); note that the time axis in the plot
is logarithmic.  The summary plot in Figure~\ref{fig-knotnodes}
aggregates all running times
for each number of tetrahedra $n$ using the geometric mean
(which benefits the older algorithms, since they have much wider variability).
Figure~\ref{fig-knotperf} shows the individual running times for every
input under each algorithm.

The results are striking.
When viewed on this log scale, the linear profiles of the older algorithms
indicate a clear exponential-time behaviour.
Moreover,
for these older algorithms we could only use the first $515$ input knots
(as sorted by increasing $n$),
because running times became too large to proceed any
further---a linear regression on $\log(\mathrm{time})$
suggests that for the largest case with $n=50$ the older tree
traversal algorithm would take $\sim 6000$ years, and the older double
description method would take $\sim 40,000$ years.  In contrast,
the new algorithm solved all $2977$ input cases (including
the largest case with $n=50$) in under 5~minutes each.


\section{3-sphere recognition and prime decomposition} \label{s-sphere}

Looking beyond unknot recognition,
we can adapt our new algorithm for other topological problems,
such as 3-sphere recognition and prime decomposition.  We briefly sketch
the key ideas here.

Given a triangulation $\tri$ that represents a closed orientable 3-manifold,
\emph{3-sphere recognition} asks whether this manifold is topologically
equivalent to the 3-sphere, and \emph{prime decomposition} breaks
this 3-manifold into a collection of ``prime factors'' (which combine to
make the original manifold using the topological operation of connected sum).

For both problems, early algorithms were mathematical breakthroughs but
algorithmically cumbersome and impractical to implement
\cite{jaco95-algorithms-decomposition,rubinstein95-3sphere}.
They have since enjoyed great improvements in both
implementability and efficiency, but like unknot recognition
they still require worst-case exponential time.
See \cite{burton13-regina} for a modern formulation of these algorithms
as they appear today.

For both algorithms, the central operations---and
their exponential-time bottle\-necks---are steps that search for normal or
``octagonal almost normal'' spheres within a triangulation $\tri$.
An \emph{octagonal almost normal} surface is like a normal surface, but
in addition to triangles and quadrilaterals we require precisely one
octagonal piece in precisely one tetrahedron.

We can adapt Algorithm~\ref{alg-search-broad} (our new search based on
branching and feasibility tests) for these tasks.
To locate normal spheres, we can essentially use
Algorithm~\ref{alg-search-broad} as it stands;
to locate almost normal spheres,
we use a variant that works with a different coordinate
system (which supports octagonal pieces).
These new search procedures can be
dropped directly into modern 3-sphere recognition and prime decomposition
implementations \cite{burton13-regina}.

\begin{figure}[tb]
    \centering
    \includegraphics[scale=0.4]{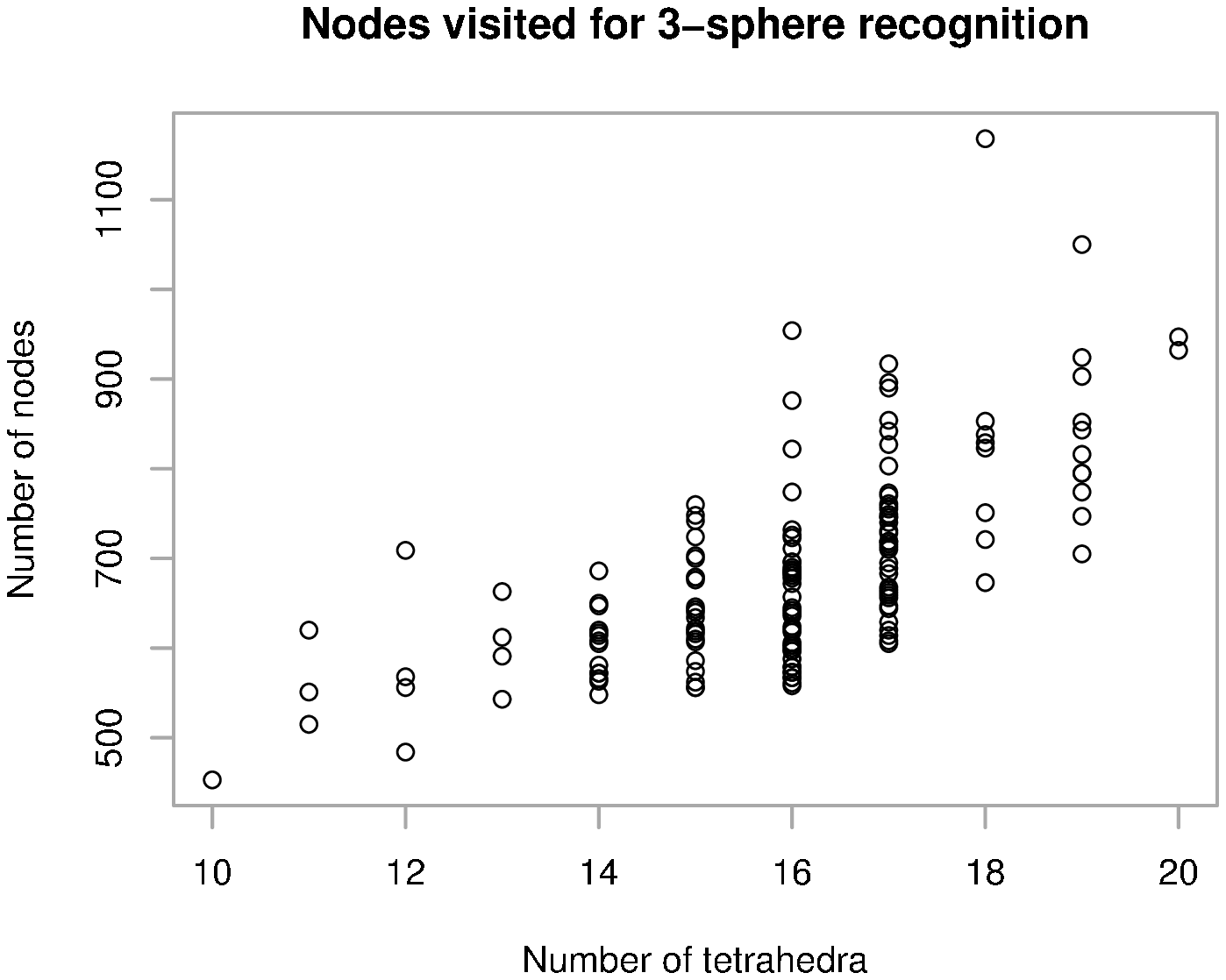} \qquad
    \includegraphics[scale=0.4]{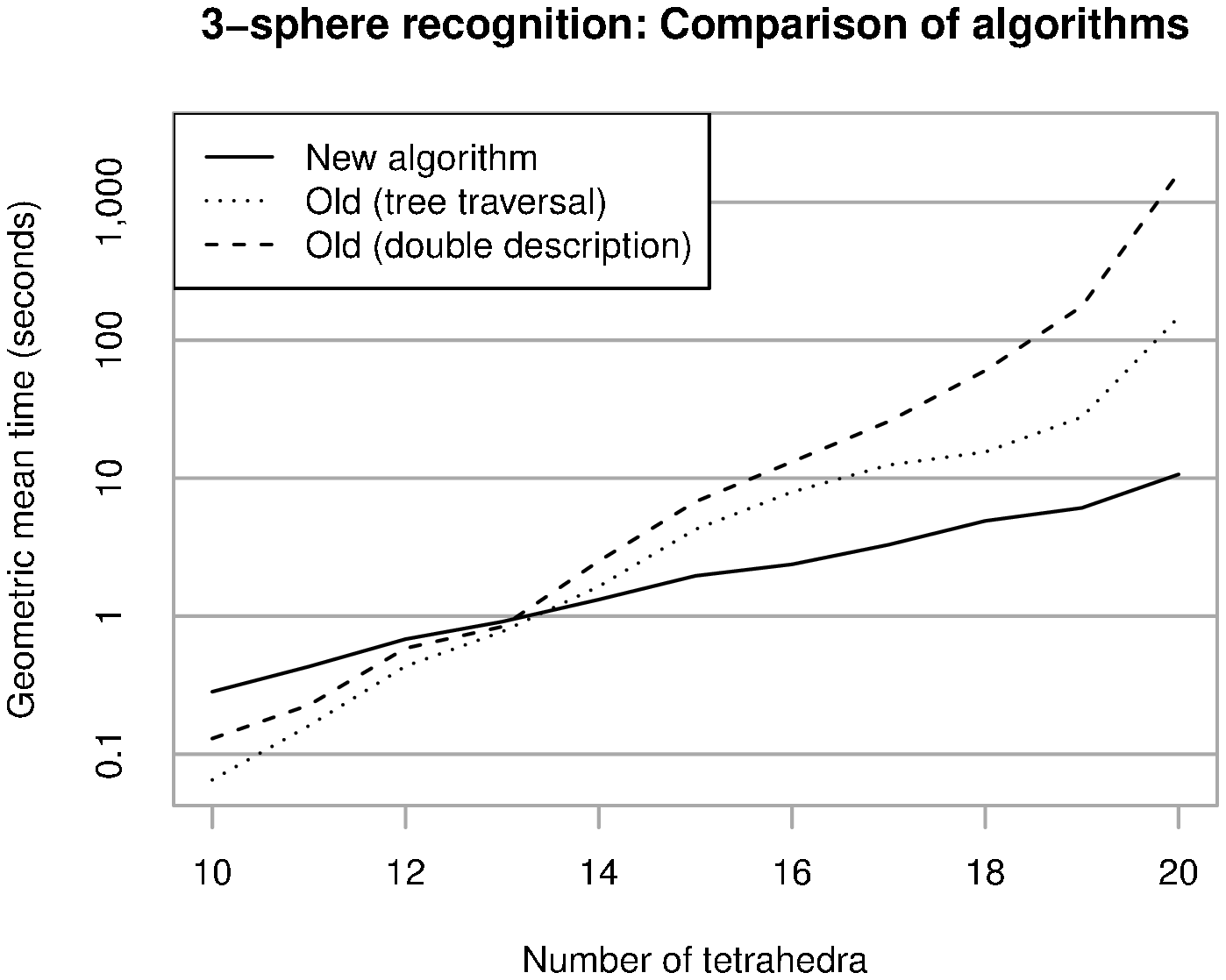}
    \caption{Performance summary for 3-sphere recognition}
    \label{fig-homnodes}
\end{figure}

As before, we empirically test our new algorithm for 3-sphere
recognition over a comprehensive and ``punishing'' data set.
This time our test inputs are the first $150$ homology spheres in the
Hodgson-Weeks census \cite{hodgson94-closedhypcensus};
again, since none of the inputs are 3-spheres, these are
difficult cases to solve: simplification tools cannot solve the problem
alone, and there is no opportunity for early termination of our
branching search.

\begin{figure}[tb]
    \centering
    \includegraphics[scale=0.4]{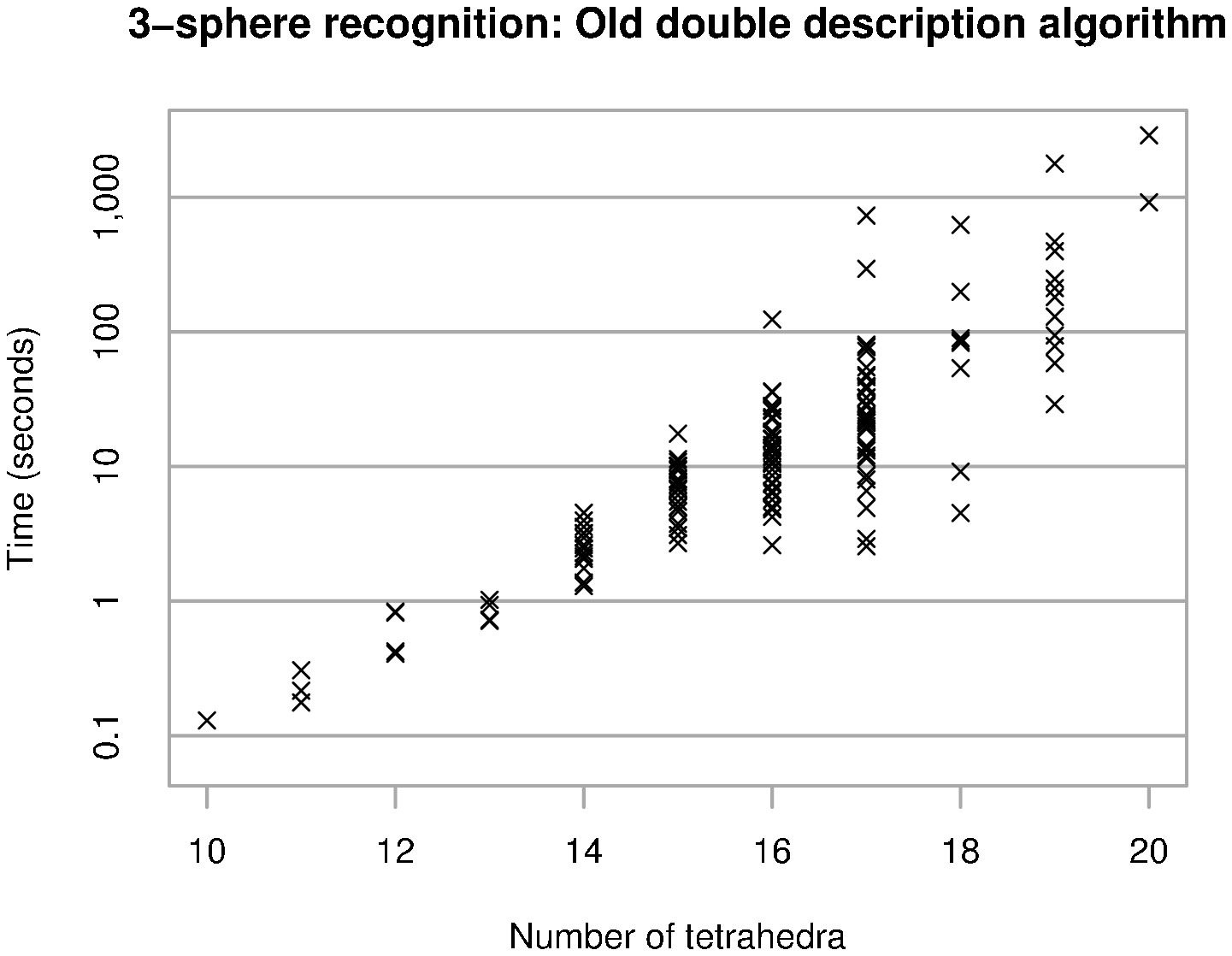} \qquad
    \includegraphics[scale=0.4]{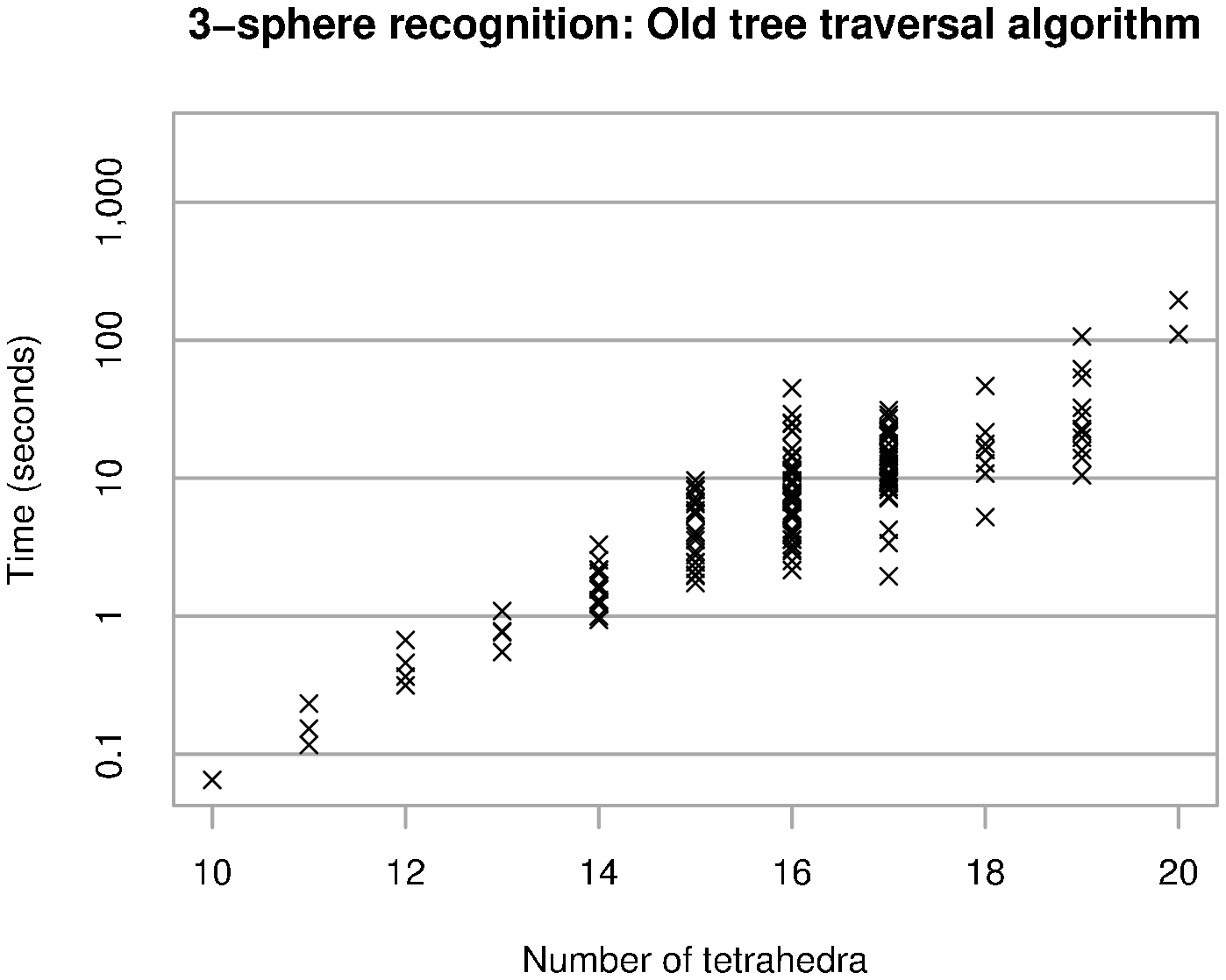} \bigskip \\
    \includegraphics[scale=0.4]{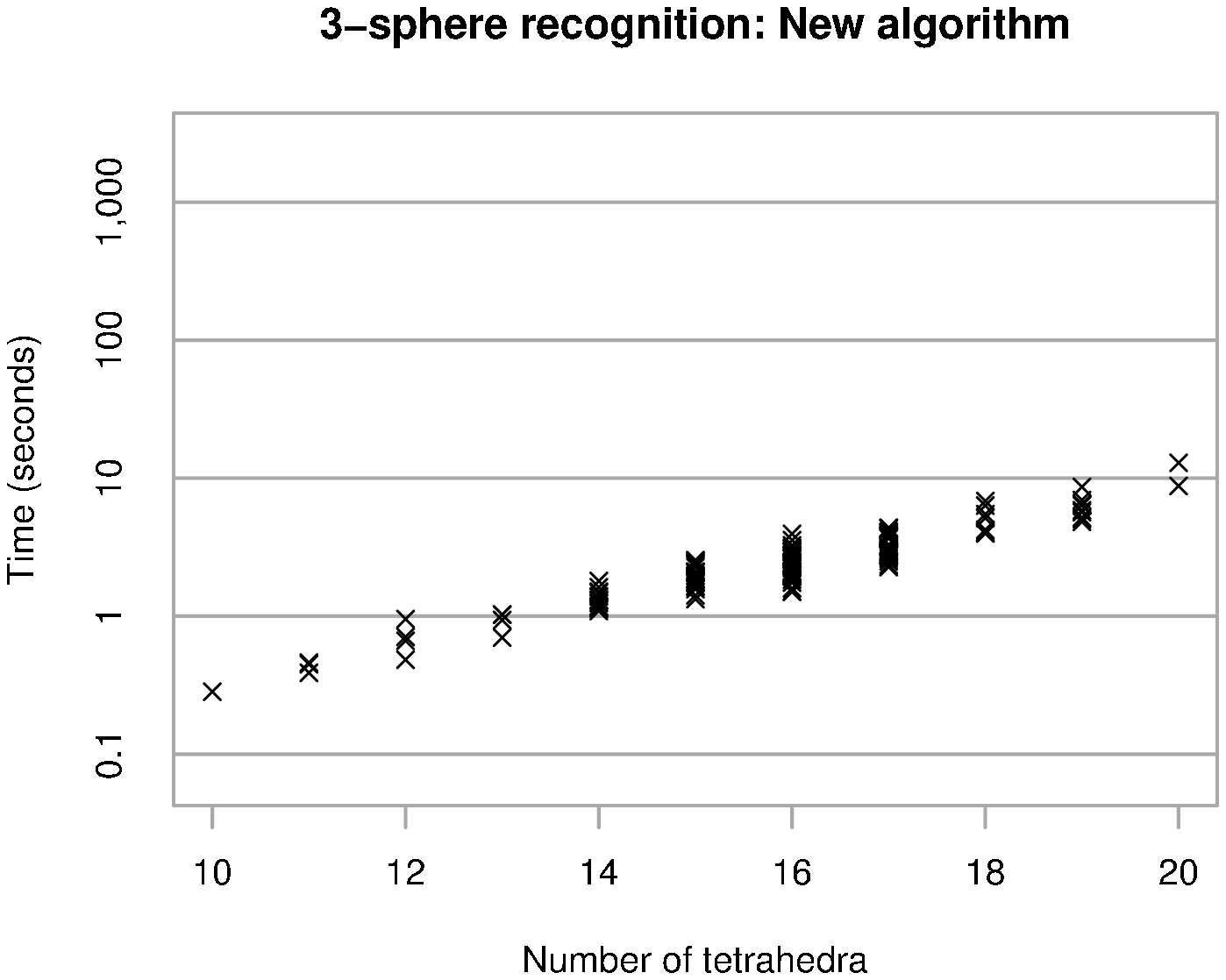}
    \caption{Detailed running times for 3-sphere recognition
        on the first 150 homology spheres}
    \label{fig-homperf}
\end{figure}

The triangulations of these homology spheres have up to $n=20$ tetrahedra.
Figure~\ref{fig-homnodes} shows the number of nodes visited in the
search tree for each input, and summarises the performance of the old
and new algorithms.  Figure~\ref{fig-homperf} shows detailed running times
for each input case.

The results for 3-sphere recognition are less clear-cut than for
unknot recognition, and no longer exhibit a polynomial-time profile;
we expect this is due to the introduction of almost normal surfaces.
Nevertheless, the results are still extremely pleasing.
The new algorithm, though slower for
small cases, becomes markedly faster than the prior algorithms as $n$ increases.
For the final case ($n=20$), the new algorithm
runs over $15$ times faster than the prior tree
traversal algorithm and roughly $223$ times faster than the prior
double description algorithm.
Furthermore, as is evident from Figure~\ref{fig-homperf},
the new algorithm exhibits much less variability in running times.

We do not test the prime decomposition algorithm separately.
This is essentially because the prime decomposition algorithm appears as
a sub-component of 3-sphere recognition (which we have already tested
above); see \cite{burton13-regina,jaco03-0-efficiency} for details.


\section{Discussion} \label{s-disc}

Although our new algorithm for unknot recognition remains exponential
time in theory, the observed polynomial growth in practice is
extremely pleasing, and indeed very exciting---this moves the
study of unknot recognition into a new phase, where we can now solve
it conclusively and quickly ``in practice'',
albeit without theoretical guarantees on the running time.  This is
reminiscent of the \emph{simplex method} for linear programming,
an algorithm that requires exponential time in the worst case but which,
despite the existence of polynomial-time alternatives
\cite{hacijan79-polylp,karmarkar84-new}, still enjoys widespread
use because of its extremely fast ``typical'' behaviour in practice.

We have a good understanding of \emph{why} the simplex
method works well in practice: it has been shown to be polynomial time in
settings such as average, generic and smoothed complexity
\cite{smale83-simplex,spielman01-smoothed,vershik83-simplex}.
In contrast, there are no such results for unknot recognition; more
generally, average, generic and smoothed results are extremely scarce
in the study of topological algorithms on 3-dimensional triangulations
(the setting for this and many other knot algorithms).
Reasons include:
\begin{itemize}
    \item Combinatorial models of random triangulations produce an
    overwhelming amount of ``junk'':
    the probability that a random pairwise gluing of faces from $n$ tetrahedra
    yields a 3-manifold triangulation tends to zero as
    \mbox{$n \to \infty$} 
    \cite{dunfield06-random-covers},
    and there is no known polynomial-time algorithm for sampling a
    random 3-manifold triangulation \cite{reiner12-mfo-problems}.

    \item ``Walking'' through the space of 3-manifold triangulations is
    difficult because the diameter of this space could be extremely large:
    for knot complements the best known bounds involve
    exponentially high towers of exponentials \cite{mijatovic05-knot}.
\end{itemize}

In Section~\ref{s-sphere} we showed how to adapt this new
unknot recognition algorithm
to the problems of 3-sphere recognition and prime decomposition.
Looking further, these techniques have potential to extend to an even broader
range of topological problems:
promising candidates include 0-efficiency algorithms \cite{jaco03-0-efficiency},
and the difficult but important problem of finding incompressible surfaces
\cite{jaco84-haken}.

Of course knots and their complements can grow significantly larger than
12 crossings and 50 tetrahedra, and it is difficult to know whether the
consistent polynomial-time behaviour seen in our experiments is
maintained as $n$ grows.  As an initial exploration into this, we have run
the same experiments over the 20-crossing \emph{dodecahedral knots}
$D_f$ and $D_s$ \cite{aitchison92-cubings}.
These are larger knots that exhibit remarkable properties
\cite{neumann92-arithmetic}, and that lie well beyond the scope of the
{\knotinfo} database.

For both $D_f$ and $D_s$ we see the same polynomial-time profile as in
our earlier experiments.  The knot complements have $n=72$ and $n=73$
tetrahedra (so the algorithm works in vector spaces of
dimension $504$ and $511$ respectively),
and for both knots Algorithm~\ref{alg-unknot}
visits precisely $8n+1$ nodes in the search tree (the smallest possible,
as discussed in Section~\ref{s-expt-results}).
Both running times are under half an hour.

It seems reasonable to believe that pathological inputs \emph{should} exist that
force Algorithm~\ref{alg-unknot} to traverse a genuinely exponential
search tree (just as, for the simplex method, it is possible to build
pathological linear programs that require an exponential number of
pivots \cite{klee72-simplex}).
However, it remains a considerable ongoing challenge to find them.

%
%

\end{document}